\documentclass[11pt,twoside]{article}
\usepackage[T1]{fontenc}
\usepackage[utf8]{inputenc}  
\usepackage{textcomp}        
\usepackage[a4paper]{geometry}
\usepackage{float}
\usepackage{wrapfig}
\usepackage{amsmath}
\usepackage{amsthm}

\newtheorem{theorem}{Theorem}[section] 
\newtheorem{lemma}[theorem]{Lemma}

\theoremstyle{remark}
\newtheorem{remark}[theorem]{Remark}

\usepackage{amssymb}
\usepackage{graphicx}
\usepackage{bm}
\usepackage{stmaryrd}
\usepackage{esint}
\usepackage[unicode=true]
 {hyperref}

\makeatletter

\providecommand{\tabularnewline}{\\}
\newcommand{\lyxdot}{.}

\@ifundefined{showcaptionsetup}{}{%
 \PassOptionsToPackage{caption=false}{subfig}}
\usepackage{subfig}
\makeatother

\begin{document}
\title{Higher-order spectral element method for the stationary Stokes interface
problem in two dimensions}

\author{Shivangi Joshi%
\thanks{Research Scholar, Department of Mathematics, BITS-Pilani Hyderabad Campus, Hyderabad, India. Email: \texttt{joshishivangi024@gmail.com}}, Kishore Kumar Naraparaju%
\thanks{Associate Professor, Department of Mathematics, BITS-Pilani Hyderabad Campus, Hyderabad, India. Email: \texttt{naraparaju@hyderabad.bits-pilani.ac.in}}, Subhashree Mohapatra%
\thanks{Assistant Professor, Department of Mathematics, IIIT Delhi, Delhi, India. Email: \texttt{subhashree@iiitd.ac.in}}%
\thanks{This work has been supported by the National Board for Higher Mathematics, India.}%
}

\date{~}
\maketitle
\begin{abstract}
This article presents a higher-order spectral element method for the
two-dimensional Stokes interface problem involving a piecewise constant
viscosity coefficient. The proposed numerical formulation is based
on least-squares formulation. The mesh is aligned with the interface,
and the interface is completely resolved using blending element functions.
The higher-order spectral element functions are nonconforming, and
the same-order approximation is used for both velocity and pressure
variables. The interface conditions are added to the minimizing functional
in appropriate Sobolev norms. Stability and error estimates are proven.
The proposed method is shown to be exponentially accurate in both
velocity and pressure variables. The theoretical estimates are validated
through various numerical examples.
\end{abstract}

\section{Introduction}

This paper uses a piecewise constant viscosity coefficient to look at the Stokes interface problem in multi-phase incompressible fluid flows. It applies to situations like emulsions, biological flows, and industrial mixing, where fluids have different densities and viscosities at interfaces. The jump in the viscosity coefficient at the interface results in jumps in the velocity and pressure. This leads to low global regularity and, therefore, makes the Stokes interface problem challenging to solve. These discontinuities, combined with the complex geometry of the interface, add to the difficulty of achieving accurate numerical approximations. Maintaining good accuracy while keeping the computational cost reasonably low is particularly challenging. The numerical approximation of the Stokes interface problem has been widely studied in the literature; however, the regularity of the solution in the individual regions is little known. 

Broadly, we can categorize the numerical methods for the interface problems into two categories: the fitted mesh methods and the unfitted mesh approach. Fitted mesh methods align the mesh with the interface, approximating it using piecewise polynomials. However, re-meshing is a requirement when the interfaces move and increases the computational complexity. Moving interfaces are made easier with unfitted mesh approaches where the mesh has holes and cuts the interface through its elements. These methods are subdivided into enrichment-based methods, which add degrees of freedom near the interface, and modified space methods, which adapt the finite element space to satisfy interface conditions without increasing degrees of freedom. Prominent examples include Nitsche's method, the Immersed Finite Element Method (IFEM), the Extended Finite Element Method (XFEM), and the Cut Finite Element Method (CutFEM). All these methods have advantages and disadvantages; sometimes, accuracy, stability, and computational cost have to be compromised. In all these methods, the finite element spaces are wisely chosen for velocity and pressure variables such that the Ladyzhenskaya--Babu\v{s}ka--Brezzi (LBB) condition is satisfied.

Several other approaches have been explored to solve Stokes interface problems. Finite difference methods, such as those by Li et al.~\cite{LI1} and Dong et al.~\cite{DoZh}, handle discontinuous viscosity along arbitrary interfaces and resolve jumps in the solution and its derivatives. A generalized finite difference method was also introduced by Shao et al.~\cite{shaosongli}. 

In the context of finite volume methods, Sevilla and Duretz~\cite{SeDu} developed a face-centered finite volume method (FCFV) that avoids the Ladyzhenskaya–Babuška–Brezzi (LBB) condition while maintaining desirable properties from hybridizable discontinuous Galerkin methods.

Finite element methods have also been extensively studied. Olshanskii and Reusken~\cite{OL1} analyzed a finite element formulation and its solver, establishing well-posedness and inf-sup results for cases with uniform viscosity jumps. Theoretical insights into the interface conditions, including pressure and velocity jumps, were also derived in \cite{IT1}. Ohmori and Saito~\cite{OH1} proved optimal error bounds for various approximations,  such as MINI element and P1-iso-P2/P1 element approximations, while Song and Gao~\cite{SO1} validated the inf-sup stability of methods for multi-subdomain cases.

A cut finite element method for a Stokes interface problem has been studied in \cite{HA1}. The authors have used Nitsche's formulation to capture interface discontinuities, and to ensure a well-conditioned matrix, stabilization terms are added for both velocity and pressure. The optimal order convergence is proved theoretically and verified numerically. 

Kirchhart et al. \cite{KI1} have addressed Stokes interface problems using $P_{1}$ extended finite element space for pressure
and standard confirming $P_{2}$ finite element space for velocity. A stabilizing term has been added for stability. An optimal preconditioner for the stiffness matrix corresponding to the pair mentioned above is presented. Error estimates are derived. Because of the standard $P_{2}$-FE velocity space, the error bound is not optimal when the normal derivative of the velocity is discontinuous across the interface.
In \cite{LE1}, an unfitted finite element discretisation based on the Taylor-Hood velocity-pressure pair with XFEM (or CutFEM) enhancements has been proposed. Nitsche's formulation
is used to implement interface conditions, and a ghost penalty stabilization is used to ensure inf-sup stability. Further, optimal error bounds are shown.

In 2015, Adjerid et al. \cite{AD} proposed an immersed $Q_{1}-Q_{0}$ discontinuous Galerkin method for Stokes interface problems, in which they considered the coupling of velocity and pressure while constructing the immersed finite element (IFE) spaces. The concept was extended by utilizing immersed $CR-P_{0}$ (where $CR$ stands for Crouzeix--Raviart) and the Rannacher--Turek rotated $Q_{1}-Q_{0}$ elements in \cite{JO1}. Numerical results with optimal convergence rates have
been provided. Chen and Zhang \cite{CH1} have proposed a Taylor-Hood-immersed finite element method for Stokes interface problems. Lower-order immersed finite element spaces $(P_{2}-P_{1})$ are used for approximation.
Numerical examples with different interface conditions and different coefficients demonstrate the optimal convergence of the method. The theoretical study for IFE approaches for Stokes interface problems has been provided in \cite{jiwangli}. The conventional $CR-P_{0}$ finite element space is modified to create IFE spaces. The unisolvence of IFE basis functions and the best approximation abilities of IFE space have been proven. Furthermore, the authors have shown that the IFE approach is stable and produces optimal error estimates. The
authors of \cite{JoKw} came up with a new $P_{1}$-nonconforming-based IFE method for Stokes interface problems that is different from the one used in \cite{AD,JO1} in terms of the modification of the basis functions. They constructed a velocity basis on the interface element,
which is less coupled to the pressure basis compared with \cite{AD} or \cite{JO1}. Another characteristic of their IFE space is that the pressure basis has two degrees of freedom on the interface element to manage the discontinuity of the pressure variable. Unfitted methods have also been studied in \cite{burmandel,huwangchen,nwangchen,WA1,WA2,ramanbhupen}.

Despite these advances, existing methods often rely on lower-order approximations or use different polynomial orders for velocity and pressure to satisfy stability conditions.
 Least-squares methods address this limitation by producing symmetric and positive-definite systems. Spectral methods, known for their high accuracy, have been combined with least-squares formulations for Stokes flows, as discussed in~\cite{BOTH, JN1, PROO2, PR003}. Hessari~\cite{He2} demonstrated spectral convergence for least-squares pseudo-spectral methods, while other works~\cite{HE1, hes3} addressed cases with singular sources and discontinuous viscosities.This problem
has also been studied using various numerical schemes in \cite{AS,CHI,He,johnansonn,Su,YA}.

The nonconforming spectral element method (NSEM) has shown promise for solving steady-state Stokes systems~\cite{kishsubham, S1, S2, S3}.  In
\cite{kishnaga,arbazsubham}, nonconforming spectral element method for elliptic interface problems in two and three dimensions has been studied, respectively. In \cite{kishore kumar}, NSEM for elasticity interface problems has been studied. More details about NSEM can be found in \cite{kishshi}.

Building on such advances, we propose a least-squares-based nonconforming spectral element
method for Stokes interface problems with smooth interfaces. Unlike classical methods, this approach ensures exponential accuracy while maintaining computational efficiency. Key features of the proposed method include the use of same-order spectral element functions for velocity and pressure, incorporation of interface conditions into the minimizing functional via appropriate Sobolev norms, and exponential convergence for both velocity and pressure variables.

The error estimates derived for velocity (in the $H^1$ norm) and pressure (in the $L^2$ norm) demonstrate exponential decay. These results are verified numerically and also highlight the mass conservation property of the scheme through analysis of the error in the continuity equation. This
paper is structured as follows: 
Section~(\ref{sec:interface-problem}) introduces the Stokes interface problem and  establishes basic regularity results. Section~(\ref{sec:discretization}) describes the nonconforming spectral element discretization and the construction of suitable basis functions. Section~(\ref{sec:stability-estimate}) presents the least-squares formulation, proves the discrete stability estimate. In Section~(\ref{sec:Numerical Formulation_ch4}), we formulate the least-squares spectral element functional that forms the basis of the NSEM discretization. This functional not only incorporates the residuals of the Stokes equations within each element but also systematically penalizes jumps in velocity, its derivatives, and pressure across inter-element boundaries and jumps across the physical interface $\Gamma_0$.
Section~(\ref{sec:ErrorEstimates_ch4}) presents the main theoretical error estimates for the method. It demonstrates that the method achieves exponential convergence in the relevant norms under regularity assumptions.
Section~(\ref{sec:NumericalResultsCh4}) presents an extensive array of numerical experiments centered on benchmark Stokes interface problems, showing exponential convergence in both velocity and pressure. We also present the error in the continuity equation to demonstrate the conservation property of the method.
 Finally, Section~(\ref{sec:conclusion}) concludes the paper with insights and future directions.

Here we give some notations and define required function spaces. Let
$\Omega\subset\mathbb{R}^{2},$ be an open bounded set with sufficiently
smooth boundary $\partial\Omega$. $H^{m}(\Omega)$ denotes the Sobolev
space of functions with square integrable derivatives of integer order
less than or equal to $m$ in $\Omega$ equipped with the norm 
\[
\left\Vert u\right\Vert _{H^{m}(\Omega)}^{2}=\sum_{\left|\alpha\right|\leq m}\left\Vert D^{\alpha}u\right\Vert _{L^{2}(\Omega)}^{2}.
\]
Further, let $I=(-1,1).$ Then we define fractional norms $(0<s<1)$
by : 
\begin{eqnarray*}
\left\Vert w\right\Vert _{s,I}^{2}=\|w\|_{0,I}^{2}+\int_{I}\int_{I}\frac{\left|w(\xi)-w(\xi^{\prime})\right|^{2}}{\left|\xi-\xi^{\prime}\right|^{1+2s}}d\xi d\xi^{\prime}.
\end{eqnarray*}
Moreover, 
\begin{align*}
\|w\|_{1+s,I}^{2}=\|w\|_{0,I}^{2}+\left\Vert \:\frac{\partial w}{\partial\xi}\:\right\Vert _{s,I}^{2}+\left\Vert \:\frac{\partial w}{\partial\eta}\:\right\Vert _{s,I}^{2}\;.
\end{align*}
We denote the vectors by bold letters. For example, $\bm{u}=(u_{1},u_{2})^{T}$,
$\mathbf{H}^{k}(\Omega)=H^{k}(\Omega)\times H^{k}(\Omega),$ etc.
The norms are given by $||\bm{u}||_{k,\Omega}^{2}=||u_{1}||_{k,\Omega}^{2}+||u_{2}||_{k,\Omega}^{2}$
for $\bm{u}\in\mathbf{H}^{k}(\Omega),$ $\left\Vert \bm{u}\right\Vert _{s,I}^{2}=||u_{1}||_{s,I}^{2}+||u_{2}||_{s,I}^{2},$
etc.

\section{Stokes interface problem }
\label{sec:interface-problem}

\begin{figure}[H]
\centering
\includegraphics[scale=0.8]{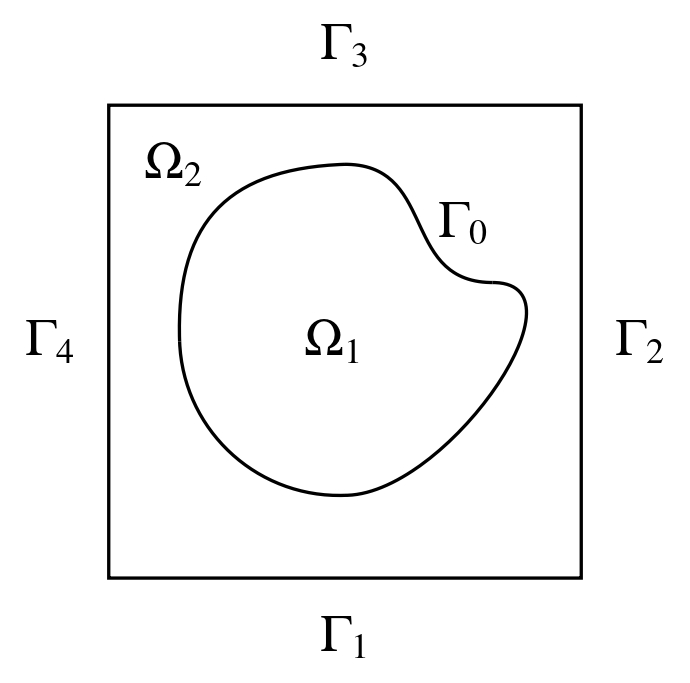}
\caption{Domain with smooth interface $\Gamma_{0}$}
\label{fig:Domain with smooth interface}
\end{figure}
Let $\Omega \subset \mathbb{R}^{2}$ be a bounded polygonal domain, and let $\Omega_1 \subset \Omega$ be an open subdomain with boundary $\Gamma_0 = \partial \Omega_1 \subset \Omega$. Define $\Omega_2 = \Omega \setminus \overline{\Omega}_1$ such that $\Omega_1 \cap \Omega_2 = \emptyset$ (see Figure~(\ref{fig:Domain with smooth interface})). Here, $\Gamma_0$ is the interface between the subdomains $\Omega_1$ and $\Omega_2$, and is assumed to be smooth, at least of class $C^2$. Let the boundary of $\Omega$ be denoted by $\Gamma = \cup_{i=1}^{4} \Gamma_i$. A point in $\Omega$ is denoted by $\bm{x} = (x, y)$.

We consider the following Stokes interface problem:
\begin{align}
\begin{cases}
-\operatorname{div}(\nu \nabla \bm{u}) + \nabla p = \bm{f} \quad &\text{in } \Omega_1 \cup \Omega_2, \\
\operatorname{div} \bm{u} = \bm{0} \quad &\text{in } \Omega_1 \cup \Omega_2, \\
\bm{u} = \bm{0} \quad &\text{on } \Gamma,
\end{cases}
\label{eq:3.1}
\end{align}
with the interface conditions
\begin{align}
\llbracket \bm{u} \rrbracket = \bm{0}, \qquad
\llbracket (\nu \nabla \bm{u} - p \, \bm{I}) \bm{n} \rrbracket = \bm{g} \quad \text{on } \Gamma_0,
\label{eq:3.2}
\end{align}
where the viscosity $\nu$ is piecewise constant:
\[
\nu =
\begin{cases}
\nu_1 > 0 &\text{in } \Omega_1, \\
\nu_2 > 0 &\text{in } \Omega_2.
\end{cases}
\]

Here, $\bm{u} = (u_1, u_2)^{T}$ is the velocity field, $\bm{f} = (f_1, f_2)^{T}$ is the force term, and $\bm{I}$ denotes the identity matrix.

We define the following differential operators used throughout the formulation:
\begin{equation}
\mathcal{L} \bm{u} := -\operatorname{div}(\nu \nabla \bm{u}) + \nabla p, \qquad
\mathcal{D} \bm{u} := -\operatorname{div} \bm{u}.
\label{eq:3.3}
\end{equation}

The velocity and pressure fields are denoted separately on the two subdomains as:
\begin{align}
\bm{u}_1 &:= \bm{u}|_{\Omega_1} = 
\begin{pmatrix}
u_1^1 \\ u_2^1
\end{pmatrix}, 
&
\bm{u}_2 &:= \bm{u}|_{\Omega_2} = 
\begin{pmatrix}
u_1^2 \\ u_2^2
\end{pmatrix}, \notag \\
p_1 &:= p|_{\Omega_1}, 
&
p_2 &:= p|_{\Omega_2}. 
\label{eq:3.4}
\end{align}

The jump in the velocity field across the interface $\Gamma_0$ is defined as
\begin{equation}
\llbracket \bm{u} \rrbracket := \bm{u}_1 - \bm{u}_2.
\label{eq:3.5}
\end{equation}

Let $\bm{n} = (n_1, n_2)^{\mathrm{T}}$ denote the unit normal vector to $\Gamma_0$ pointing outward from $\Omega_1$ into $\Omega_2$. Then, the jump in the stress across the interface is defined by
\begin{equation}
\llbracket (\nu \nabla \bm{u} - p \, \bm{I}) \bm{n} \rrbracket :=
(\nu_1 \nabla \bm{u}_1 - p_1 \bm{I}) \, \bm{n}
- (\nu_2 \nabla \bm{u}_2 - p_2 \bm{I}) \, \bm{n},
\label{eq:3.6}
\end{equation}

where $\bm{I}$ denotes the $2 \times 2$ identity matrix.

\subsubsection*{Uniqueness and Regularity}

We define the broken Sobolev space over the union $\Omega_1 \cup \Omega_2$ as
\[
\bm{H}^{k}(\Omega_1 \cup \Omega_2) :=
\left\{ \bm{u} \in \bm{L}^{2}(\Omega) \,\middle|\, 
\bm{u}|_{\Omega_i} \in \bm{H}^{k}(\Omega_i), \; i = 1, 2 \right\},
\]
with the norm
\[
\| \bm{u} \|_{k, \Omega_1 \cup \Omega_2} :=
\| \bm{u}_1 \|_{k, \Omega_1} + \| \bm{u}_2 \|_{k, \Omega_2}.
\]

We assume the pressure variable satisfies the compatibility condition
\[
\int_{\Omega} p \, \bm{dx} = 0.
\]
Let $\bm{f} \in \bm{L}^{2}(\Omega)$ and $\bm{g} \in \bm{H}^{1/2}(\Gamma_0)$. Then the variational formulation of the Stokes interface problem \eqref{eq:3.1}-\eqref{eq:3.2} admits a unique weak solution
\[
(\bm{u}, p) \in \bm{H}_0^{1}(\Omega) \times L^{2}(\Omega).
\]

Moreover, the solution satisfies the following regularity result (see~\cite{shibata, WA1}):
\[
\bm{u} \in \bm{H}^{2}(\Omega_1 \cup \Omega_2), \quad
p \in H^{1}(\Omega_1 \cup \Omega_2).
\]

\begin{theorem}[Subdomain Regularity Estimate]
\label{thm:regularity}
Assume that $\bm{f} \in \bm{L}^{2}(\Omega)$ and $\bm{g} \in \bm{H}^{1/2}(\Gamma_0)$. Then the weak solution $(\bm{u}, p)$ to the Stokes interface problem~\eqref{eq:3.1}–\eqref{eq:3.2} satisfies the estimate
\[
\| \bm{u} \|_{2, \Omega_1 \cup \Omega_2} + \| p \|_{1, \Omega_1 \cup \Omega_2}
\leq C \left( \| \bm{f} \|_{0, \Omega} + \| \bm{g} \|_{1/2, \Gamma_0} \right),
\]
where $C > 0$ is a constant independent of the discretization.
\end{theorem}

\begin{remark}[Global vs. Subdomain Regularity]
Unlike scalar elliptic interface problems where global and subdomain regularity are well characterized, the theory for Stokes interface problems is less complete. While it is known that $\bm{u} \in \bm{H}^2(\Omega_i)$ and $p \in H^1(\Omega_i)$ for $i = 1,2$ under sufficient smoothness of $\Gamma_0$ and the data, global $H^2$-regularity across $\Omega$ is generally not guaranteed. This lack of regularity stems from the discontinuity in viscosity and associated stress, which induces singular behavior across $\Gamma_0$. The identification and treatment of such singularities remain an active area of research.
\end{remark}

Based on the regularity theory of the Stokes problem and Theorem~(\ref{thm:regularity}) (see also \cite{kishsubham,ramanbhupen}), we obtain:

\begin{equation}
\left\Vert \bm{u} \right\Vert_{2, \Omega_1 \cup \Omega_2}
+ \left\Vert p \right\Vert_{1, \Omega_1 \cup \Omega_2}
\leq C \left(
\left\Vert \bm{f} \right\Vert_{0, \Omega}
+ \left\Vert \mathcal{D} \bm{u} \right\Vert_{1, \Omega}
+ \left\Vert \llbracket \bm{u} \rrbracket \right\Vert_{3/2, \Gamma_0}
+ \left\Vert \bm{g} \right\Vert_{1/2, \Gamma_0}
+ \left\Vert \bm{u} \right\Vert_{3/2, \Gamma}
\right).
\label{eq:3.7}
\end{equation}
Here again, $C$ is independent of mesh parameters but may depend on the domain and regularity of the interface.

\section{Discretization and Spectral Element Functions}
\label{sec:discretization}
\begin{figure}[H]
\centering
\includegraphics[scale=0.8]{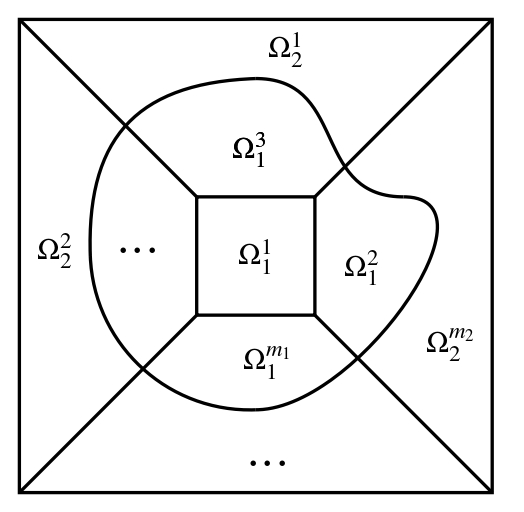}
\caption{Interface-fitted mesh discretization}
\label{fig:Interface-fitted mesh discretization}
\end{figure}

The subdomains $\Omega_1$ and $\Omega_2$ are partitioned into a finite number of curvilinear quadrilateral elements:
\[
\Omega_1 = \bigcup_{l=1}^{m_1} \Omega_1^l, \qquad \Omega_2 = \bigcup_{l=1}^{m_2} \Omega_2^l,
\]
such that the discretization matches on the interface (see Figure~(\ref{fig:Interface-fitted mesh discretization})). Every triangle can be subdivided into quadrilaterals, making this a general approach.

Define analytical invertible bilinear maps $M_i^l: {S} \rightarrow \Omega_i^l$, where the reference square is given by
\[
{S} = (-1,1)^2 = \{ (\xi, \eta) \in \mathbb{R}^2 \mid -1 < \xi < 1, -1 < \eta < 1 \},
\]
such that
\[
(x, y) = M_i^l(\xi, \eta) = \left(X_i^l(\xi, \eta), Y_i^l(\xi, \eta)\right), \quad \text{for } i=1,2,\ \text{ and, } l=1,\dots,m_i.
\]
These mappings can be constructed using transfinite interpolation techniques (see \cite{GO}) and are assumed to be smooth and invertible.

The spectral element functions $\{ \widetilde{\bm{u}}_1^l \}_l$ and $\{ \widetilde{\bm{u}}_2^l \}_l$ are defined as tensor-product polynomials of degree \( W \) in the reference coordinates \( (\xi, \eta) \):
\begin{equation}
\begin{aligned}
\widetilde{\bm{u}}_1^l(\xi, \eta) 
&= 
\begin{pmatrix}
\left(\tilde{u}_{1}^{1,l}(\xi,\eta),\tilde{u}_{2}^{1,l}(\xi,\eta)\right)^{T}
\end{pmatrix}
=
\begin{pmatrix}
\displaystyle \sum_{r=0}^W \sum_{s=0}^W g_{r,s}^{1,l} \xi^r \eta^s \\
\displaystyle \sum_{r=0}^W \sum_{s=0}^W g_{r,s}^{2,l} \xi^r \eta^s
\end{pmatrix}, \\[2ex]
\widetilde{\bm{u}}_2^l(\xi, \eta) 
&= 
\begin{pmatrix}
\left(\tilde{u}_{1}^{2,l}(\xi,\eta),\tilde{u}_{2}^{2,l}(\xi,\eta)\right)^{T}
\end{pmatrix}
=
\begin{pmatrix}
\displaystyle \sum_{r=0}^W \sum_{s=0}^W h_{r,s}^{1,l} \xi^r \eta^s \\
\displaystyle \sum_{r=0}^W \sum_{s=0}^W h_{r,s}^{2,l} \xi^r \eta^s
\end{pmatrix}.
\end{aligned}
\label{eq:3.8}
\end{equation}

These are simply Lagrange interpolating polynomials defined on the Gauss--Lobatto--Legendre (GLL) points.

The corresponding physical spectral element functions $\{ \bm{u}_1^l \}_l$ and $\{ \bm{u}_2^l \}_l$ are obtained by composing with the inverse of the element mappings:
\begin{equation}
\bm{u}_1^l(x, y) = \widetilde{\bm{u}}_1^l \circ (M_1^l)^{-1}(x, y), \qquad
\bm{u}_2^l(x, y) = \widetilde{\bm{u}}_2^l \circ (M_2^l)^{-1}(x, y).
\label{eq:3.9}
\end{equation}

Similarly, for the pressure field, the reference pressure functions are defined as:
\begin{equation}
\tilde{p}_i^l(\xi, \eta) = \sum_{r=0}^W \sum_{s=0}^W b_{r,s}^{i,l} \, \xi^r \eta^s, \quad\text{for } i=1,2,\ l=1,\dots,m_i.
\label{eq:3.10}
\end{equation}

The corresponding physical pressure fields \( p_i^l \) in \( \Omega_i \) are then given by:
\begin{equation}
p_i^l(x, y) = \tilde{p}_i^l \circ (M_i^l)^{-1}(x, y).
\label{eq:3.11}
\end{equation}

All spectral element functions are assumed to be nonconforming; that is, continuity is not enforced across inter-element boundaries. Notably, the method does not require different polynomial orders for velocity and pressure approximations. This flexibility is a key feature of the nonconforming spectral element method (NSEM) framework.

\section{Stability Estimate}
\label{sec:stability-estimate}
For \(i=1,2\) and \(l=1,\dots,m_i\), we define
\[
\mathcal{L}(\bm{u}_{i}^{l}, p_{i}^{l}) = -\operatorname{div}(\nu_{i} \nabla \bm{u}_{i}^{l}) + \nabla p_{i}^{l},
\quad \text{and} \quad
\mathcal{D}(\bm{u}_{i}^{l}) = -\nabla \cdot \bm{u}_{i}^{l}.
\]

Let \( J_{i}^{l}(\xi, \eta) \) denote the Jacobian of the mapping \( M_{i}^{l}(\xi, \eta) \) from \( S = (-1,1)^2 \) to \( \Omega_{i}^{l} \), for \( i = 1, 2 \) and \( l = 1, \dots, m_i \).
 Then,
\begin{equation}
\int_{\Omega_{i}^{l}} \left| \mathcal{L}(\bm{u}_{i}^{l}, p_{i}^{l}) \right|^2 \, dx \, dy =
\int_{S} \left| \mathcal{L}(\widetilde{\bm{u}}_{i}^{l}, \tilde{p}_{i}^{l}) \right|^2 J_{i}^{l} \, d\xi \, d\eta.
\label{eq:3.12}
\end{equation}

Define $\mathcal{L}_{i}^{l}(\widetilde{\bm{u}}_{i}^{l}, \tilde{p}_{i}^{l}) = \mathcal{L}(\widetilde{\bm{u}}_{i}^{l}, \tilde{p}_{i}^{l}) \sqrt{J_{i}^{l}}$. Then,
\begin{equation}
\int_{\Omega_{i}^{l}} \left| \mathcal{L}(\bm{u}_{i}^{l}, p_{i}^{l}) \right|^2 \, dx \, dy =
\int_{S} \left| \mathcal{L}_{i}^{l}(\widetilde{\bm{u}}_{i}^{l}, \tilde{p}_{i}^{l}) \right|^2 \, d\xi \, d\eta.
\label{eq:3.13}
\end{equation}

Similarly, define \(\mathcal{D}_{i}^{l} \widetilde{\bm{u}}_{i}^{l} = \mathcal{D}(\widetilde{\bm{u}}_{i}^{l}) \sqrt{J_{i}^{l}}\).
We approximate the coefficients in \(\mathcal{L}_{i}^{l}(\widetilde{\bm{u}}_{i}^{l}, \tilde{p}_{i}^{l})\) and \(\mathcal{D}_{i}^{l} \widetilde{\bm{u}}_{i}^{l}\) by polynomials of degree \(W\), obtained by orthogonal projection onto \(H^2(S)\). The resulting operators are denoted by \((\mathcal{L}_{i}^{l})^a(\widetilde{\bm{u}}_{i}^{l}, \tilde{p}_{i}^{l})\) and \((\mathcal{D}_{i}^{l})^a \widetilde{\bm{u}}_{i}^{l}\), respectively.

\begin{remark}
It can be shown that the error in the approximation is spectrally small. To keep the presentation simple, we omit the error term wherever it appears in the theory in the following sections.
\end{remark}

\begin{figure}[H]
\centering
\includegraphics[scale=0.5]{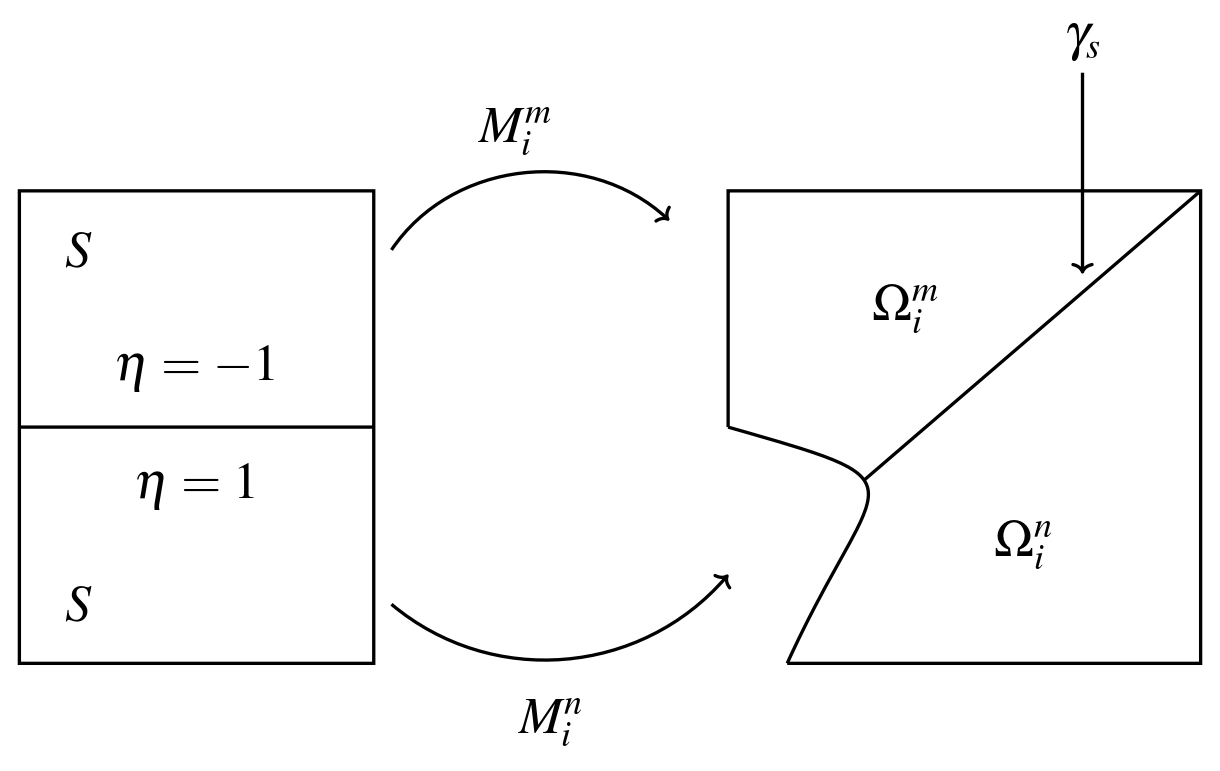}
\caption{Inter-element boundary transfer}
\label{fig:Interelement_transfer}
\end{figure}
To define the stability functional, let \(\gamma_s\) denote a side shared by adjacent elements \(\Omega_{i}^{m}\) and \(\Omega_{i}^{n}\), for \(i = 1, 2\); (see Figure~(\ref{fig:Interelement_transfer})). Assume that \(\gamma_s\) is the image of \(\eta = -1\) under \(M_{i}^{m}\), and of \(\eta = 1\) under \(M_{i}^{n}\). 

Let \(\hat{\xi}_x(\xi,\eta)\) denote the polynomial approximation of \(\xi_x(\xi,\eta)\) of degree \(W\) in \(\xi\) and \(\eta\) separately, as defined in Theorem~(4.46) of~\cite{schwab}. In a similar way, \(\hat{\xi}_y\), \(\hat{\eta}_x\), and \(\hat{\eta}_y\) can be defined.

Then, by the chain rule, we have:
\begin{align}
(\bm{u}_{i}^{m})^a_x &= (\widetilde{\bm{u}}_{i}^{m})_{\xi} \, \hat{\xi}_x + (\widetilde{\bm{u}}_{i}^{m})_{\eta} \, \hat{\eta}_x, \label{eq:3.14} \\
(\bm{u}_{i}^{m})^a_y &= (\widetilde{\bm{u}}_{i}^{m})_{\xi} \, \hat{\xi}_y + (\widetilde{\bm{u}}_{i}^{m})_{\eta} \, \hat{\eta}_y. \label{eq:3.15}
\end{align}

Define inter-element jumps:
\begin{align}
\| \llbracket \bm{u}_{i} \rrbracket \|_{0, \gamma_s}^2 &= \| \widetilde{\bm{u}}_{i}^{m}(\xi, -1) - \widetilde{\bm{u}}_{i}^{n}(\xi, 1) \|_{0, I}^2,
\label{eq:3.16} \\
\| \llbracket (\bm{u}_{i})^a_x \rrbracket \|_{1/2, \gamma_s}^2 &= \| (\bm{u}_{i}^{m})^a_x(\xi, -1) - (\bm{u}_{i}^{n})^a_x(\xi, 1) \|_{1/2, I}^2,
\label{eq:3.17} \\
\| \llbracket (\bm{u}_{i})^a_y \rrbracket \|_{1/2, \gamma_s}^2 &= \| (\bm{u}_{i}^{m})^a_y(\xi, -1) - (\bm{u}_{i}^{n})^a_y(\xi, 1) \|_{1/2, I}^2,
\label{eq:3.18} \\
\| \llbracket (p_{i}) \rrbracket \|_{1/2, \gamma_s}^2 &= \| (\widetilde{p}_{i}^{m})(\xi, -1) - (\widetilde{p}_{i}^{n})(\xi, 1) \|_{1/2, I}^2.
\label{eq:3.19}
\end{align}
Here and in what follows, \(I\) is an interval \((-1,1)\).

\begin{figure}[H]
\centering
\includegraphics[scale=0.5]{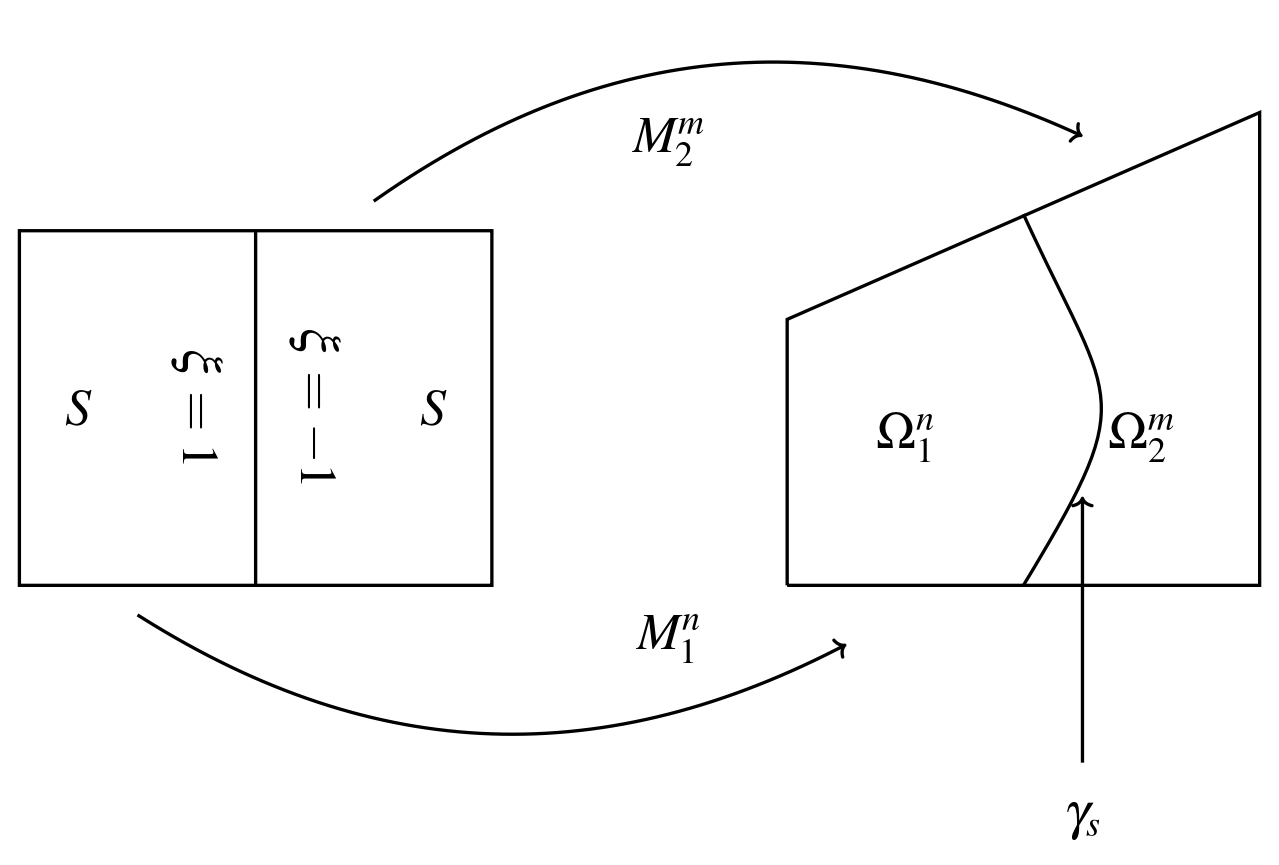}
\caption{Interface boundary transfer}
\label{fig:Interface_transfer}
\end{figure}
Now, we define the jump across the interface $\Gamma_0$. Let $\gamma_s \subseteq \Gamma_0$ be a shared edge between elements $\Omega_{1}^{n}$ and $\Omega_{2}^{m}$, with $\gamma_s$ mapped from $\xi = 1$ in $\Omega_{1}^{n}$ and $\xi = -1$ in $\Omega_{2}^{m}$ (see Figure~ (\ref{fig:Interface_transfer})). Define the interface jump as
\begin{equation}
\begin{aligned}
\| \llbracket \bm{u} \rrbracket \|_{3/2, \gamma_s}^2 &= \| \bm{u}_1 - \bm{u}_2 \|_{3/2, \gamma_s}^2 \\
&= \| \widetilde{\bm{u}}_{1}^{n}(1, \eta) - \widetilde{\bm{u}}_{2}^{m}(-1, \eta) \|_{0, I}^2 \\
&\quad + \left\| \left(\frac{\partial \widetilde{\bm{u}}_{1}^{n}}{\partial T}\right)^a(1, \eta) - \left(\frac{\partial \widetilde{\bm{u}}_{2}^{m}}{\partial T}\right)^a(-1, \eta) \right\|_{1/2, I}^2
\end{aligned}
\label{eq:3.20}
\end{equation}
where \(\frac{\partial \widetilde{\bm{u}}_{1}^{n}}{\partial T}\) and \(\frac{\partial \widetilde{\bm{u}}_{2}^{m}}{\partial T}\) are the tangential derivatives of \(\widetilde{\bm{u}}_{1}^{n}\) and \(\widetilde{\bm{u}}_{2}^{m}\), respectively, and their approximations, as defined earlier, are given by
\[
\left( \frac{\partial \widetilde{\bm{u}}_{1}^{n}}{\partial T} \right)^a
\quad \text{and} \quad
\left( \frac{\partial \widetilde{\bm{u}}_{2}^{m}}{\partial T} \right)^a.
\]

The stress terms are:
\begin{equation}
\begin{aligned}
\left( \nu_{1} \nabla \bm{u}_{1}^{n} - p_{1}^{n} \bm{I} \right)^a \bm{n}
&=
\begin{pmatrix}
\left( \nu_{1} \left(\frac{\partial u_{1}^{1,n}}{\partial x}\right)^a - \widetilde{p}^{n}_{1} \right) n_{1}
+ \nu_{1} \left(\frac{\partial u_{1}^{1,n}}{\partial y}\right)^a \, n_{2} \\[2.1ex]
\nu_{1} \left(\frac{\partial u_{2}^{1,n}}{\partial x}\right)^a \, n_{1}
+ \left( \nu_{1} \left(\frac{\partial u_{2}^{1,n}}{\partial y}\right)^a - \widetilde{p}_{1}^{n} \right) n_{2}
\end{pmatrix},
\\[2ex]
\left( \nu_{2} \nabla \bm{u}_{2}^{m} - p_{2}^{m} \bm{I} \right)^a \bm{n}
&=
\begin{pmatrix}
\left( \nu_{2} \left(\frac{\partial u_{1}^{2,m}}{\partial x}\right)^a - \widetilde{p}_{2}^{m} \right) n_{1}
+ \nu_{2} \left(\frac{\partial u_{1}^{2,m}}{\partial y}\right)^a \, n_{2} \\[2.1ex]
\nu_{2} \left(\frac{\partial u_{2}^{2,m}}{\partial x}\right)^a \, n_{1}
+ \left( \nu_{2} \left(\frac{\partial u_{2}^{2,m}}{\partial y}\right)^a - \widetilde{p}_{2}^{m} \right) n_{2}
\end{pmatrix}.
\end{aligned}
\label{eq:3.21}
\end{equation}

Using \eqref{eq:3.14}-\eqref{eq:3.15} define the norm of stress jump across \(\gamma_s \subset \Gamma_0\) as:
\begin{align}
\left\Vert \llbracket (\nu \nabla \bm{u} - p \bm{I})^a \bm{n} \rrbracket \right\Vert_{1/2, \gamma_s}^2
= \left\Vert
\left( (\nu_1 \nabla \bm{u}_1^n - p_1^n \bm{I})^a \bm{n} \right)(1, \eta)
- \left( (\nu_2 \nabla \bm{u}_2^m - p_2^m \bm{I})^a \bm{n} \right)(-1, \eta)
\right\Vert_{1/2, I}^2.
 \label{eq:3.22}
\end{align}

Now, along the boundary $\Gamma = \cup_{j=1}^{4} \Gamma_j$ (see Figure~(\ref{fig:Domain with smooth interface})), let $\gamma_s \subseteq \Gamma_j$ (for some $j$) be the image of $\xi = 1$ under the mapping $M_2^m$ from $S$ to $\Omega_2^m$. Then,
\begin{align}
\left\Vert \bm{u}_2 \right\Vert_{0, \gamma_s}^2
+ \left\Vert \Bigl(\frac{\partial \bm{u}_2}{\partial T}\Bigr)^a \right\Vert_{1/2, \gamma_s}^2
= \left\Vert \widetilde{\bm{u}}_2^m(1, \eta) \right\Vert_{0, I}^2
+ \left\Vert \Bigl(\frac{\partial \widetilde{\bm{u}}_2^m}{\partial T}\Bigr)^a(1, \eta) \right\Vert_{1/2, I}^2.
\label{eq:3.23}
\end{align}

Let \(\bm{\lambda} = (\xi, \eta)\), and let \(\{ \mathcal F_{\widetilde{\bm{u}},\widetilde p} \}\) denote the spectral element representation of the function \((\bm{u}, p)\), i.e.,
\[
\{ \mathcal F_{\widetilde{\bm{u}},\widetilde p} \} =
\left\{
\big\{ \widetilde{\bm{u}}_1^k(\bm{\lambda}) \big\}_k,\,
\big\{ \tilde{p}_1^k(\bm{\lambda}) \big\}_k,\,
\big\{ \widetilde{\bm{u}}_2^l(\bm{\lambda}) \big\}_l,\,
\big\{ \tilde{p}_2^l(\bm{\lambda}) \big\}_l
\right\}.
\]
Let
\begin{equation}
\Pi^W = \left(
\{ \widetilde{\bm{u}}_1^k(\bm{\lambda} ) \}_k,\,
\{ \tilde{p}_1^k(\bm{\lambda} ) \}_k,\,
\{ \widetilde{\bm{u}}_2^l(\bm{\lambda} ) \}_l,\,
\{ \tilde{p}_2^l(\bm{\lambda} ) \}_l
\right)
 \label{eq:3.24}
\end{equation}
denote the space of spectral element functions.

All spectral element functions are assumed to be nonconforming in the functional sense; that is, continuity is not enforced across inter-element boundaries or across the interface \(\Gamma_0\).

We define \( \mathcal F_i^{int} \) as the set of all interior edges of \( \Omega_i \), i.e.,
\[
\mathcal F_i^{int}
\;:=\;
\bigl\{
  \gamma \subset \bar{\Omega}_{i}\setminus\partial\Omega_i:\gamma = \partial\Omega_i^{\,n_1}\cap\partial\Omega_i^{\,n_2}
  \bigm|\;
  n_1\neq n_2
\bigr\}
\]
 each edge \( \gamma_l \in \mathcal F_i^{int} \) is a common edge shared between two distinct elements \( \Omega_i^{n_1} \) and \( \Omega_i^{n_2} \) of \( \Omega_i \). 
We define the stability functional \(\mathcal{V}^W\left( \{ \mathcal F_{\widetilde{\bm{u}},\widetilde p} \} \right)\) as:
\begin{align}
\mathcal{V}^W\left( \{ \mathcal F_{\widetilde{\bm{u}},\widetilde p} \} \right) &=\,
\sum_{k=1}^{m_1} \left\Vert (\mathcal{L}_1^k)^a(\widetilde{\bm{u}}_1^k, \tilde{p}_1^k) \right\Vert_{0, S}^2
+ \sum_{l=1}^{m_2} \left\Vert (\mathcal{L}_2^l)^a(\widetilde{\bm{u}}_2^l, \tilde{p}_2^l) \right\Vert_{0, S}^2 \nonumber \\
&+ \sum_{k=1}^{m_1} \left\Vert (\mathcal{D}_1^k)^a \widetilde{\bm{u}}_1^k \right\Vert_{1, S}^2
+ \sum_{l=1}^{m_2} \left\Vert (\mathcal{D}_2^l)^a \widetilde{\bm{u}}_2^l \right\Vert_{1, S}^2 \nonumber \\
&+ \sum_{i=1}^{2} \sum_{\gamma_l \in \mathcal F_i^{int}} \bigg(
\left\Vert \left\llbracket \bm{u}_i \right\rrbracket \right\Vert_{0, \gamma_l}^2
+ \left\Vert \left\llbracket \left(  \bm{u}_i \right)_x^a \right\rrbracket \right\Vert_{1/2, \gamma_l}^2
+ \left\Vert \left\llbracket \left(  \bm{u}_i \right)_y^a \right\rrbracket \right\Vert_{1/2, \gamma_l}^2  + \left\Vert \left\llbracket p_i \right\rrbracket \right\Vert_{1/2, \gamma_l}^2
\bigg) \nonumber \\
&+ \sum_{\gamma_s \subseteq \Gamma_0} \bigg(
\left\Vert \left\llbracket \bm{u} \right\rrbracket \right\Vert_{3/2, \gamma_s}^2
+ \left\Vert \left\llbracket \left( (\nu \nabla \bm{u} - p \bm{I}) \bm{n} \right)^a \right\rrbracket \right\Vert_{1/2, \gamma_s}^2
\bigg) \nonumber \\
&+ \sum_{\gamma_s \subseteq \Gamma} \bigg(
\left\Vert \bm{u}_2 \right\Vert_{0, \gamma_s}^2
+ \left\Vert \left( \frac{\partial \bm{u}_2}{\partial T} \right)^a \right\Vert_{1/2, \gamma_s}^2
\bigg).
\label{eq:3.25}
\end{align}

\begin{lemma}\label{lem:velocity-correction}
Let
\(
\left\{ \left\{ \widetilde{\bm{u}}_{1}^{k}(\xi,\eta) \right\}_{k}, \left\{ \widetilde{\bm{u}}_{2}^{l}(\xi,\eta) \right\}_{l} \right\} \in \Pi^{W}.
\)
Then there exists
\(
\left\{ \left\{ \tilde{\bm{v}}_{1}^{k}(\xi,\eta) \right\}_{k}, \left\{ \tilde{\bm{v}}_{2}^{l}(\xi,\eta) \right\}_{l} \right\},
\)\\
where \( \tilde{\bm{v}}_{1}^{k} = \bm{0} \) on \( \Gamma_0 \text{ }\forall \text{ }k = 1,2,\dots,m_1 \), and \( \tilde{\bm{v}}_{2}^{l} = \bm{0} \) on \( \Gamma_0 \text{ } \forall \text{ }  l = 1,2,\dots,m_2 \).

such that
\[
\bm{v}_{1}^{k} \in \bm{H}^{2}(\Omega_{1}^{k}), \quad \bm{v}_{2}^{l} \in \bm{H}^{2}(\Omega_{2}^{l}),
\]
and
\[
\bm{w}_1 = \bm{u}_1 + \bm{v}_1 \in \bm{H}^{2}(\Omega_1), \quad \bm{w}_2 = \bm{u}_2 + \bm{v}_2 \in \bm{H}^{2}(\Omega_2).
\]

Moreover, the following estimate holds:
\begin{align}
& \sum_{k=1}^{m_1} \left\Vert \tilde{\bm{v}}_1^k(\xi,\eta) \right\Vert_{2,S}^2 
+ \sum_{l=1}^{m_2} \left\Vert \tilde{\bm{v}}_2^l(\xi,\eta) \right\Vert_{2,S}^2 \nonumber \\
& \leq c (\ln W)^2 
\sum_{i=1}^{2} \sum_{\gamma_l \in \mathcal F_i^{int}} \Big(
\left\Vert \llbracket \bm{u}_i \rrbracket \right\Vert_{0,\gamma_l}^2
+ \left\Vert \llbracket (\bm{u}_i)^a_x \rrbracket \right\Vert_{1/2,\gamma_l}^2
+ \left\Vert \llbracket (\bm{u}_i)^a_y \rrbracket \right\Vert_{1/2,\gamma_l}^2
\Big).
\label{eq:3.26}
\end{align}
\end{lemma}

\begin{proof}
We give a rough sketch of the proof. For the more details of the proof, see Lemma~(7.1) in~\cite{pd1}.

Since
\begin{align*}
&\bigl\{ \{\widetilde{\bm{u}}_{1}^{k}(\xi,\eta)\}_{k}, \{\widetilde{\bm{u}}_{2}^{l}(\xi,\eta)\}_{l} \bigr\}
\end{align*}
are non-conforming, we apply corrections
\begin{align*}
&\bigl\{ \{\tilde{\bm{r}}_{1}^{k}(\xi,\eta)\}_{k}, \{\tilde{\bm{r}}_{2}^{l}(\xi,\eta)\}_{l} \bigr\}
\end{align*}
at the common vertices of the elements in \( \Omega_1 \) and \( \Omega_2 \) such that
\[
\tilde{\bm{r}}_1^k = \bm{0} \quad \text{and} \quad \tilde{\bm{r}}_2^l = \bm{0} \quad \text{on } \Gamma_0.
\]

Consider an element \( \Omega_i^k \subset \Omega_i \). There exists an analytical invertible map \( M_{i}^{k} : S \rightarrow \Omega_{i}^{k} \), as defined in Section~(\ref{sec:discretization}), such that \( S = (M_{i}^{k})^{-1}(\Omega_{i}^{k}) \).

Let \( P_j \), for \( j = 1, \ldots, 4 \), denote the vertices of the reference square \( S \), and define
\[
\tilde{\bm{r}}_1^k(\xi,\eta) = 
\begin{pmatrix}
\tilde{r}_1^{1,k} \\
\tilde{r}_2^{1,k}
\end{pmatrix}, \quad
\tilde{\bm{r}}_2^l(\xi,\eta) = 
\begin{pmatrix}
\tilde{r}_1^{2,l} \\
\tilde{r}_2^{2,l}
\end{pmatrix}.
\]

We construct the set of polynomials
\[
\left\{ \left\{ \tilde{\bm{r}}_{1}^{k}(\xi,\eta) \right\}_{k}, \left\{ \tilde{\bm{r}}_{2}^{l}(\xi,\eta) \right\}_{l} \right\}
\]
such that, for \( j = 1, \dots, 4 \) and \( i = 1, 2 \),
\begin{align*}
\left( \widetilde{\bm{u}}_i^k + \tilde{\bm{r}}_i^k \right)(P_j) &= \bar{\bm{u}}(P_j), \\
\left( (\widetilde{\bm{u}}_i^k)_x + (\tilde{\bm{r}}_i^k)_x \right)(P_j) &= \bar{\bm{u}}_x(P_j), \\
\left( (\widetilde{\bm{u}}_i^k)_y + (\tilde{\bm{r}}_i^k)_y \right)(P_j) &= \bar{\bm{u}}_y(P_j)
\end{align*}
Here, \( \bar{z} \) denotes the average of the values of \( z \) at \( P_j \) over all elements that have \( P_j \) as a vertex. Moreover, \( \tilde{\bm{r}}_{i}^{k} \) is a polynomial of degree less than or equal to four, and it satisfies the estimate
\begin{equation}
\left\Vert \tilde{\bm{r}}_{i}^{k}(\xi,\eta) \right\Vert_{2, S}^{2}
\leq 
C \sum_{l=1}^{4} \left( \left| \bm{a}_{l} \right|^{2} + \left| \bm{b}_{l} \right|^{2} + \left| \bm{c}_{l} \right|^{2} \right), \quad \text{for } i = 1, 2
\label{eq:3.27}
\end{equation}

where
\[
\begin{aligned}
\tilde{\bm{r}}_{i}^{k}(P_l) &= \bm{a}_l = 
\begin{pmatrix}
a_{1,l} \\
a_{2,l}
\end{pmatrix}, \\
\left( \tilde{\bm{r}}_{i}^{k} \right)_x(P_l) &= \bm{b}_l = 
\begin{pmatrix}
b_{1,l} \\
b_{2,l}
\end{pmatrix}, \\
\left( \tilde{\bm{r}}_{i}^{k} \right)_y(P_l) &= \bm{c}_l = 
\begin{pmatrix}
c_{1,l} \\
c_{2,l}
\end{pmatrix}, \quad \text{for } l = 1, \dots, 4;i = 1,2.
\end{aligned}
\]

\[
\begin{aligned}
\left| \bm{a}_{l} \right|^{2} &= \left| a_{1,l} \right|^{2} + \left| a_{2,l} \right|^{2}, \\
\left| \bm{b}_{l} \right|^{2} &= \left| b_{1,l} \right|^{2} + \left| b_{2,l} \right|^{2}, \\
\left| \bm{c}_{l} \right|^{2} &= \left| c_{1,l} \right|^{2} + \left| c_{2,l} \right|^{2}.
\end{aligned}
\]
Using the inequality (Theorem 4.79 of \cite{schwab})
\[
\left\Vert q \right\Vert_{L^{\infty}(\bar{I})}^{2} \leq C \, (\ln W) \left\Vert q \right\Vert_{1/2, I}^{2},
\]
where \( C \) is a constant and \( q(y) \) is a polynomial of degree \( W \) defined on \( I = (-1,1) \), we obtain:
\[
\sum_{k=1}^{m_1} \| \tilde{\bm{r}}_1^k \|_{2,S}^2 + \sum_{l=1}^{m_2} \| \tilde{\bm{r}}_2^l \|_{2,S}^2 
\leq K (\ln W)^2 
\sum_{i=1}^{2} \sum_{\gamma_l \in \mathcal F_i^{int}} \Big(
\left\Vert \llbracket \bm{u}_i \rrbracket \right\Vert_{0,\gamma_l}^2
+ \left\Vert \llbracket (\bm{u}_i)^a_x \rrbracket \right\Vert_{1/2,\gamma_l}^2
+ \left\Vert \llbracket (\bm{u}_i)^a_y \rrbracket \right\Vert_{1/2,\gamma_l}^2
\Big).
\]

Let
\[
\tilde{\bm{y}}_{1}^{k}(\xi,\eta) = \widetilde{\bm{u}}_{1}^{k}(\xi,\eta) + \tilde{\bm{r}}_{1}^{k}(\xi,\eta), \quad
\tilde{\bm{y}}_{2}^{l}(\xi,\eta) = \widetilde{\bm{u}}_{2}^{l}(\xi,\eta) + \tilde{\bm{r}}_{2}^{l}(\xi,\eta).
\]

Now we define a correction
\[
\left\{ \left\{ \tilde{\bm{s}}_{1}^{k}(\xi,\eta) \right\}_{k}, \left\{ \tilde{\bm{s}}_{2}^{l}(\xi,\eta) \right\}_{l} \right\}
\]
on the edges of \( S = (M_{i}^{m})^{-1}(\Omega_{i}^{m}) \) for all \( i \) and \( m \), such that
\[
\tilde{\bm{s}}_{1}^{k} = \bm{0} \quad \text{and} \quad \tilde{\bm{s}}_{2}^{l} = \bm{0} \quad \text{on } \Gamma_0,
\]
with \( \tilde{\bm{s}}_{1}^{k}, \tilde{\bm{s}}_{2}^{l} \in \bm{H}^{2}(S) \) and,

\begin{equation}
\begin{aligned}
(\tilde{\bm{y}}_{i}^{k} + \tilde{\bm{s}}_{i}^{k})(t) &= \bar{\bm{y}}(t), \\
(\tilde{\bm{y}}_{i}^{k})_x + (\tilde{\bm{s}}_{i}^{k})_x(t) &= \bar{\bm{y}}_x(t), \\
(\tilde{\bm{y}}_{i}^{k})_y + (\tilde{\bm{s}}_{i}^{k})_y(t) &= \bar{\bm{y}}_y(t).
\end{aligned}
\label{eq:3.28}
\end{equation}

Where, \(t\) is a point on the side \(\gamma_l\) of \(S\) and \(\bar{\bm{y}}\) represents the averages of the values \(\widetilde{\bm{y}}_i^l\) at \(t\) over all elements which have \(\gamma_l\) as a common side. Similarly, \(\bar{\bm{y}}_x,\bar{\bm{y}}_y\) represents average values of \((\widetilde{\bm{y}}_i^l)_x,(\widetilde{\bm{y}}_i^l)_y\) respectively.



Further details can be found in \cite{pd1}.

Using Theorem 1.5.2.4 of~ \cite{grisvard} , Theorem 4.82 in
\cite{schwab} , we conclude:
\[
\sum_{k=1}^{m_1} \left\Vert \tilde{\bm{s}}_1^k \right\Vert_{2,S}^2 + \sum_{l=1}^{m_2} \left\Vert \tilde{\bm{s}}_2^l \right\Vert_{2,S}^2 
\leq K (\ln W)^2 
\sum_{i=1}^{2} \sum_{\gamma_l \in \mathcal F_i^{\text{int}}} \Big(
\left\Vert \llbracket \bm{u}_i \rrbracket \right\Vert_{0,\gamma_l}^2
+ \left\Vert \llbracket (\bm{u}_i)^a_x \rrbracket \right\Vert_{1/2,\gamma_l}^2
+ \left\Vert \llbracket (\bm{u}_i)^a_y \rrbracket \right\Vert_{1/2,\gamma_l}^2
\Big).
\]

Finally, define
\[
\tilde{\bm{v}}_{1}^{k}(\xi,\eta) = \tilde{\bm{y}}_{1}^{k}(\xi,\eta) + \tilde{\bm{s}}_{1}^{k}(\xi,\eta), \quad
\tilde{\bm{v}}_{2}^{l}(\xi,\eta) = \tilde{\bm{y}}_{2}^{l}(\xi,\eta) + \tilde{\bm{s}}_{2}^{l}(\xi,\eta).
\]
Now, the proof of Lemma~(\ref{lem:velocity-correction}) follows from the above inequalities.
\end{proof}

\begin{lemma}\label{lem:pressure-correction}
Let 
\[
\left\{ \left\{ \tilde{p}_{1}^{k}(\xi,\eta) \right\}_{k}, \left\{ \tilde{p}_{2}^{l}(\xi,\eta) \right\}_{l} \right\} \in \Pi^{W}.
\]
Then there exists 
\[
\left\{ \left\{ \tilde{o}_{1}^{k}(\xi,\eta) \right\}_{k}, \left\{ \tilde{o}_{2}^{l}(\xi,\eta) \right\}_{l} \right\}
\]
such that
\[
\tilde{o}_{1}^{k} = 0 \text{ on } \Gamma_{0} \quad \text{for all } k=1,2,\ldots,m_1,\quad
\tilde{o}_{2}^{l} = 0 \text{ on } \Gamma_{0} \quad \text{for all } l=1,2,\ldots,m_2,
\]
and
\[
o_{1}^{k} \in H^{1}(\Omega_{1}^{k}), \quad o_{2}^{l} \in H^{1}(\Omega_{2}^{l}), \quad 
p_{1}^{*} = p_{1} + o_{1} \in H^{1}(\Omega_{1}), \quad p_{2}^{*} = p_{2} + o_{2} \in H^{1}(\Omega_{2}).
\]

Moreover, the estimate
\begin{equation}
\sum_{k=1}^{m_{1}} \left\Vert \tilde{o}_{1}^{k}(\xi,\eta) \right\Vert_{1,S}^{2}
+
\sum_{l=1}^{m_{2}} \left\Vert \tilde{o}_{2}^{l}(\xi,\eta) \right\Vert_{1,S}^{2}
\leq
C (\ln W)^{2} \left( \sum_{i=1}^{2} \sum_{\gamma_s \in \mathcal F_i^{\text{int}}} \left\Vert \llbracket p_i \rrbracket \right\Vert_{\frac{1}{2}, \gamma_s}^{2} \right),
\label{eq:3.29}
\end{equation}
holds.
\end{lemma}

\begin{proof}
The proof is very similar to Lemma A.3 of~\cite{S1}.
\end{proof}

\begin{lemma}\label{lem:trace-boundary}
Let \( \bm{w}_i = \bm{u}_i + \bm{v}_i \in \bm{H}^{2}(\Omega_i) \) for \( i = 1, 2 \), where
\[
\left\{ \left\{ \widetilde{\bm{u}}_{1}^{k}(\xi, \eta) \right\}_k, \left\{ \widetilde{\bm{u}}_{2}^{l}(\xi, \eta) \right\}_l \right\} \in \Pi^W
\]
and 
\[
\left\{ \left\{ \tilde{\bm{v}}_{1}^{k}(\xi, \eta) \right\}_k, \left\{ \tilde{\bm{v}}_{2}^{l}(\xi, \eta) \right\}_l \right\}
\]
is as defined in Lemma~(\ref{lem:velocity-correction}). Then the following estimate holds:
\begin{align}
\left\Vert \bm{w}_2 \right\Vert_{\frac{3}{2}, \Gamma}^{2}
\;\leq\; & \; C\,(\ln W)^2 \Bigg(
\sum_{\gamma_s \subseteq \Gamma}
\Big(
\left\Vert \bm{u}_2 \right\Vert_{0, \gamma_s}^2
+
\left\Vert \left(\tfrac{\partial \bm{u}_2}{\partial T}\right)^a \right\Vert_{\frac{1}{2}, \gamma_s}^2
\Big) \notag \\
& \quad +
\sum_{i=1}^{2} \sum_{\gamma_l \in \mathcal{F}_i^{\mathrm{int}}}
\Big(
\left\Vert \llbracket \bm{u}_i \rrbracket \right\Vert_{0,\gamma_l}^2
+ \left\Vert \llbracket (\bm{u}_i)^a_x \rrbracket \right\Vert_{\frac{1}{2},\gamma_l}^2
+ \left\Vert \llbracket (\bm{u}_i)^a_y \rrbracket \right\Vert_{\frac{1}{2},\gamma_l}^2
\Big)
\Bigg).
\label{eq:3.30}
\end{align}

\end{lemma}

\begin{proof}
The result follows from Lemma~7.2 in~\cite{pd1}.
\end{proof}
Let
\[
\bigl\{\mathcal F_{\widetilde{\bm u},\widetilde p}\bigr\}
=
\Bigl\{
  \{\widetilde{\bm u}_{1}^{\,l}(\xi,\eta)\}_{l},\;
  \{\tilde p_{1}^{\,l}(\xi,\eta)\}_{l},\;
  \{\widetilde{\bm u}_{2}^{\,k}(\xi,\eta)\}_{k},\;
  \{\tilde p_{2}^{\,k}(\xi,\eta)\}_{k}
\Bigr\}
\in\Pi^{W}.
\]
Define the quadratic form
\(\mathcal{Q}^{W}\bigl(\mathcal F_{\widetilde{\bm u},\widetilde p}\bigr)\) by
\begin{equation}\label{eq:quadratic}
\begin{aligned}
\mathcal{Q}^{W}\bigl(\mathcal F_{\widetilde{\bm u},\widetilde p}\bigr)
:=\;&
\sum_{k=1}^{m_1}
    \bigl\|
        \widetilde{\bm u}_{1}^{\,k}(\xi,\eta)
    \bigr\|_{2,S}^{2}
+
\sum_{l=1}^{m_2}
    \bigl\|
        \widetilde{\bm u}_{2}^{\,l}(\xi,\eta)
    \bigr\|_{2,S}^{2} \\[0.5em]
&+
\sum_{k=1}^{m_1}
    \bigl\|
        \tilde p_{1}^{\,k}(\xi,\eta)
    \bigr\|_{1,S}^{2}
+
\sum_{l=1}^{m_2}
    \bigl\|
        \tilde p_{2}^{\,l}(\xi,\eta)
    \bigr\|_{1,S}^{2}.
\end{aligned}
\end{equation}

\begin{theorem}[Stability Estimate]\label{thm:stability}
For \( W \) large enough, there exists a constant \( C > 0 \) such that
\begin{equation} 
\mathcal{Q}^W(\{\mathcal F_{\widetilde{\bm{u}},\widetilde p}\}) \leq C (\ln W)^2 \, \mathcal{V}^W(\{\mathcal F_{\widetilde{\bm{u}},\widetilde p}\}), \quad \forall \mathcal F_{\widetilde{\bm{u}},\widetilde p} \in \Pi^W.
\label{eq:3.32}
\end{equation}
\end{theorem}

\begin{proof}
By Lemma~(\ref{lem:velocity-correction}), there exists
\[
\left\{ \left\{ \tilde{\bm{v}}_{1}^{k}(\xi,\eta) \right\}_{k}, \left\{ \tilde{\bm{v}}_{2}^{l}(\xi,\eta) \right\}_{l} \right\}
\]
such that \( \tilde{\bm{v}}_{1}^{k} = \bm{0} \) on \( \Gamma_0 \) for all \( k = 1, 2, \dots, m_1 \), and \( \tilde{\bm{v}}_{2}^{l} = \bm{0} \) on \( \Gamma_0 \) for all \( l = 1, 2, \dots, m_2 \), and
\[
\bm{w}_i = \bm{u}_i + \bm{v}_i \in \bm{H}^{2}(\Omega_i), \quad \text{for } i = 1, 2.
\]
Moreover, \( \bm{w}_1 = \bm{u}_1 \) and \( \bm{w}_2 = \bm{u}_2 \) on the interface \( \Gamma_0 \).

Similarly, Lemma~(\ref{lem:pressure-correction}), there exists
\[
\left\{ \left\{ \tilde{o}_{1}^{k}(\xi,\eta) \right\}_{k}, \left\{ \tilde{o}_{2}^{l}(\xi,\eta) \right\}_{l} \right\}
\]
such that \( \tilde{o}_{1}^{k} = 0 \) on \( \Gamma_0 \) for all \( k = 1, 2, \dots, m_1 \), and \( \tilde{o}_{2}^{l} = 0 \) on \( \Gamma_0 \) for all \( l = 1, 2, \dots, m_2 \), and
\[
p_{1}^{*} = p_{1} + o_{1} \in H^{1}(\Omega_{1}), \quad
p_{2}^{*} = p_{2} + o_{2} \in H^{1}(\Omega_{2}).
\]

Hence, by the regularity result (Eq.~\eqref{eq:3.7}) stated in Section~(\ref{sec:interface-problem}), we have
\begin{align}
& \left\Vert \bm{w}_{1} \right\Vert_{2,\Omega_{1}}^{2}
+ \left\Vert \bm{w}_{2} \right\Vert_{2,\Omega_{2}}^{2}
+ \left\Vert p_{1}^{*} \right\Vert_{1,\Omega_{1}}^{2}
+ \left\Vert p_{2}^{*} \right\Vert_{1,\Omega_{2}}^{2} \nonumber \\
& \leq C \Big(
\left\Vert \mathcal{L}_{1}(\bm{w}_{1}, p_{1}^{*}) \right\Vert_{0,\Omega_{1}}^{2}
+ \left\Vert \mathcal{L}_{2}(\bm{w}_{2}, p_{2}^{*}) \right\Vert_{0,\Omega_{2}}^{2}
+ \left\Vert \mathcal{D}_{1}(\bm{w}_{1}) \right\Vert_{1,\Omega_{1}}^{2}
+ \left\Vert \mathcal{D}_{2}(\bm{w}_{2}) \right\Vert_{1,\Omega_{2}}^{2} \nonumber \\
& \quad + \left\Vert \bm{w}_{2} \right\Vert_{3/2,\Gamma}^{2}
+ \left\Vert \llbracket \bm{u} \rrbracket \right\Vert_{3/2,\Gamma_{0}}^{2}
+ \left\Vert \llbracket (\nu \nabla \bm{u} - p \bm{I}) \bm{n} \rrbracket \right\Vert_{1/2,\Gamma_{0}}^{2}
\Big).
\label{eq:3.33}
\end{align}
Using Minkowski's inequality, we can write the following estimate:
\begin{align}
& \left\Vert \mathcal{L}_{1}(\bm{w}_{1}, p_{1}^{*}) \right\Vert_{0, \Omega_{1}}^{2}
+ \left\Vert \mathcal{L}_{2}(\bm{w}_{2}, p_{2}^{*}) \right\Vert_{0, \Omega_{2}}^{2} \nonumber \\
& \leq C \left(
\sum_{k=1}^{m_{1}} \left\Vert (\mathcal{L}_{1}^{k})^a(\widetilde{\bm{u}}_{1}^{k}, \tilde{p}_{1}^{k})(\xi,\eta) \right\Vert_{0, S}^{2}
+ \sum_{l=1}^{m_{2}} \left\Vert (\mathcal{L}_{2}^{l})^a(\widetilde{\bm{u}}_{2}^{l}, \tilde{p}_{2}^{l})(\xi,\eta) \right\Vert_{0, S}^{2}
\right) \nonumber \\
& \quad + c \left(
\sum_{k=1}^{m_{1}} \left\Vert \tilde{\bm{v}}_{1}^{k}(\xi,\eta) \right\Vert_{2, S}^{2}
+ \sum_{l=1}^{m_{2}} \left\Vert \tilde{\bm{v}}_{2}^{l}(\xi,\eta) \right\Vert_{2, S}^{2}
+ \sum_{k=1}^{m_{1}} \left\Vert \tilde{o}_{1}^{k}(\xi,\eta) \right\Vert_{1, S}^{2}
+ \sum_{l=1}^{m_{2}} \left\Vert \tilde{o}_{2}^{l}(\xi,\eta) \right\Vert_{1, S}^{2}
\right).
\label{eq:3.34}
\end{align}
Similarly,
\begin{align}
& \left\Vert \mathcal{D}_{1}(\bm{w}_{1}) \right\Vert_{1,\Omega_{1}}^{2}
+ \left\Vert \mathcal{D}_{2}(\bm{w}_{2}) \right\Vert_{1,\Omega_{2}}^{2} \nonumber \\
& \leq C \left(
\sum_{k=1}^{m_{1}} \left\Vert (\mathcal{D}_{1}^{k})^a \widetilde{\bm{u}}_{1}^{k} \right\Vert_{1, S}^{2}
+ \sum_{l=1}^{m_{2}} \left\Vert (\mathcal{D}_{2}^{l})^a \widetilde{\bm{u}}_{2}^{l} \right\Vert_{1, S}^{2}
\right)
+ c \left(
\sum_{k=1}^{m_{1}} \left\Vert \tilde{\bm{v}}_{1}^{k}(\xi,\eta) \right\Vert_{2, S}^{2}
+ \sum_{l=1}^{m_{2}} \left\Vert \tilde{\bm{v}}_{2}^{l}(\xi,\eta) \right\Vert_{2, S}^{2}
\right).
\label{eq:3.35}
\end{align}
The rest of the proof follows from Lemmas~(\ref{lem:velocity-correction}), (\ref{lem:pressure-correction}), and (\ref{lem:trace-boundary}).
\end{proof}
\section{Numerical Formulation}
\label{sec:Numerical Formulation_ch4}

Here we define the least-squares minimizing functional based on the stability
estimate derived in the previous section. 

Define \( \bm{f}_{1} = \bm{f}\mid_{\Omega_{1}} \) and \( \bm{f}_{2} = \bm{f}\mid_{\Omega_{2}} \).
Let \( J_{1}^{l}(\xi,\eta) \) be the Jacobian of the mapping \( M_{1}^{l}(\xi,\eta) \)
from \( S = (-1,1)^{2} \) to \( \Omega_{1}^{l} \), for \( l = 1,2,\dots,m_{1} \).
Let \( \bm{\widetilde{f}}_{1}^{l}(\xi,\eta) = \bm{f}_{1}(M_{1}^{l}(\xi,\eta)) \) and define
\[
\bm{F}_{1}^{l}(\xi,\eta) = \bm{\widetilde{f}}_{1}^{l}(\xi,\eta)\sqrt{J_{1}^{l}(\xi,\eta)},
\]
for \( l = 1,2,\dots,m_{1} \). Similarly, one can define \( \bm{F}_{2}^{k}(\xi,\eta) \) 
for \( k = 1,2,\dots,m_{2} \).

Let us suppose that on the interface \( \Gamma_{0} \), we have \( \left\llbracket \bm{u} \right\rrbracket = \bm{0} \) and \( \left\llbracket (\nu \nabla \bm{u} - p \bm{I}) \bm{n} \right\rrbracket = \bm{g} \). Let \( \gamma_{s} \subseteq \Gamma_{0} \) be the image of \( \xi = 1 \) under the mapping \( M_{1}^{n} \), which maps \( S \) to \( \Omega_{1}^{n} \), and also the image of \( \xi = -1 \) under the mapping \( M_{2}^{m} \), which maps \( S \) to \( \Omega_{2}^{m} \). These mappings are defined in Section~(\ref{sec:stability-estimate}).\\
Let
\(
\bm{l}^{m,n}(\eta) := \bm{g}^{m}(-1, \eta) = \bm{g}^{n}(1, \eta),  \text{ for } \eta \in [-1, 1].
\)\\
Further, as defined earlier, \(\bm{u}_1 = \bm{u}|_{\Omega_1}\) and \(\bm{u}_2 = \bm{u}|_{\Omega_2}\).  
Let \(\widehat{\bm{F}}_i^k\) denote the unique polynomial approximation of degree \(2W\) in the variables \(\xi\) and \(\eta\), obtained by orthogonal projection with respect to the usual inner product in \(\bm{H}^2(S)\).\\
We define \(\widehat{\bm{l}}^{m,n}\) to be the polynomial approximation of degree \(2W\), which is the orthogonal projection of \(\bm{l}^{m,n}\) with respect to the usual inner product in \(\bm{H}^2(-1,1)\).\\
Let \(\{\mathcal F_{\widetilde{\bm{u}},\widetilde p}\}=\left\{\left\{ \bm{\widetilde{u}}_{1}^{k}(\xi,\eta) \right\}_{k},
  \left\{ \widetilde{p}_{1}^{k}(\xi,\eta) \right\}_{k},
  \left\{ \bm{\widetilde{u}}_{2}^{l}(\xi,\eta) \right\}_{l},
  \left\{ \widetilde{p}_{2}^{l}(\xi,\eta) \right\}_{l}\right\} \in \Pi^W\). Then we
define the functional
\begin{align}
\mathcal R^W(\{\mathcal F_{\widetilde{\bm{u}},\widetilde{p}}\}) &=\,\,
\sum_{k=1}^{m_{1}} \left\| (\mathcal{L}_{1}^{k})^a(\bm{\widetilde{u}}_{1}^{k}, \widetilde{p}_{1}^{k}) - \widehat{\bm{F}}_{1}^{k} \right\|_{0,S}^{2}
+ \sum_{l=1}^{m_{2}} \left\| (\mathcal{L}_{2}^{l})^a(\bm{\widetilde{u}}_{2}^{l}, \widetilde{p}_{2}^{l}) - \widehat{\bm{F}}_{2}^{l} \right\|_{0,S}^{2} \nonumber \\
& + \sum_{k=1}^{m_{1}} \left\| (\mathcal{D}_{1}^{k})^a\bm{\widetilde{u}}_{1}^{k} \right\|_{1,S}^{2}
+ \sum_{l=1}^{m_{2}} \left\| (\mathcal{D}_{2}^{l})^a \bm{\widetilde{u}}_{2}^{l} \right\|_{1,S}^{2} \nonumber \\
& + \sum_{i=1}^{2} \sum_{\gamma_{s} \in \mathcal F_i^{\text{int}}} \bigg(
      \left\| \llbracket \bm{u}_{i} \rrbracket \right\|_{0,\gamma_{s}}^{2}
    + \left\| \llbracket (\bm{u}_{i})_{x}^a \rrbracket \right\|_{\frac{1}{2},\gamma_{s}}^{2}
    + \left\| \llbracket (\bm{u}_{i})_{y}^a \rrbracket \right\|_{\frac{1}{2},\gamma_{s}}^{2}
    + \left\| \llbracket p_{i} \rrbracket \right\|_{\frac{1}{2},\gamma_{s}}^{2}
    \bigg) \nonumber \\
& + \sum_{\gamma_{s} \subseteq \Gamma_{0}} \bigg(
      \left\| \llbracket \bm{u} \rrbracket  \right\|_{\frac{3}{2},\gamma_{s}}^{2}
    + \left\| (\llbracket (\nu \nabla \bm{u} - p \bm{I}) \bm{n})^a \rrbracket - \widehat{\bm{l}}^{m,n} \right\|_{\frac{1}{2},\gamma_{s}}^{2}
    \bigg) \nonumber \\
&\quad + \sum_{\Gamma_l \subset \Gamma}
\bigg(
\left\| \bm{u}_2  \right\|_{0, \Gamma_l}^2
+ \left\| \left( \frac{\partial \bm{u}_2}{\partial \bm{T}} \right)^a  \right\|_{\frac{1}{2}, \Gamma_l}^2
\bigg).
\label{eq:4.1}
\end{align}

The approximate solution is chosen as the unique
\[\{\mathcal F_{\widetilde{\bm{z}},\widetilde q}\}=
\left\{
\left\{ \bm{\widetilde{z}}_{1}^{k}(\xi,\eta) \right\}_{k},
\left\{ \widetilde{q}_{1}^{k}(\xi,\eta) \right\}_{k},
\left\{ \bm{\widetilde{z}}_{2}^{l}(\xi,\eta) \right\}_{l},
\left\{ \widetilde{q}_{2}^{l}(\xi,\eta) \right\}_{l}
\right\} \in \Pi^{W}
\]
which minimizes the functional
\(
\mathcal R^{W} \left(\{\mathcal F_{\widetilde{\bm{u}},\widetilde p}\}\right)
\) over all such \(\mathcal F_{\widetilde{\bm{u}},\widetilde p} \in \Pi^W\).

The minimization problem leads to a system of equations of the form
\begin{equation}
A V = h. \label{eq:4.2}
\end{equation}
Here, \( A \) is a symmetric, positive-definite matrix, and the vector
\( V \) consists of the values of the spectral element functions at the Gauss–Legendre–Lobatto (GLL) points.
This system is solved using the Preconditioned Conjugate Gradient Method (PCGM).
Since each iteration of the PCGM requires evaluating the matrix-vector product,
this action is computed accurately and efficiently without explicitly storing the matrix \( A \).

Define the quadratic form \(\mathcal U^W(\{\mathcal F_{\widetilde{\bm{u}},\widetilde p}\})\) as:
\begin{equation}
\mathcal{U}^{W}(\{\mathcal F_{\widetilde{\bm{u}},\widetilde p}\}) =
\sum_{k=1}^{m_{1}} \left\| \bm{\widetilde{u}}_{1}^{k}(\xi,\eta) \right\|_{2,S}^{2}
+ \sum_{k=1}^{m_{1}} \left\| \widetilde{p}_{1}^{k}(\xi,\eta) \right\|_{1,S}^{2}
+ \sum_{l=1}^{m_{2}} \left\| \bm{\widetilde{u}}_{2}^{l}(\xi,\eta) \right\|_{2,S}^{2}
+ \sum_{l=1}^{m_{2}} \left\| \widetilde{p}_{2}^{l}(\xi,\eta) \right\|_{1,S}^{2}.
\label{eq:4.3}
\end{equation}

 Using trace theorem for Sobolev spaces, the following result holds:
\begin{equation}
  \mathcal V^W(\{\mathcal F_{\widetilde{\bm{u}},\widetilde p}\}) \le K\,\mathcal{U}^{W}(\{\mathcal F_{\widetilde{\bm{u}},\widetilde p}\}).
 \label{eq:4.4}
\end{equation}
where, K is a constant. By Theorem~(\ref{thm:stability}), we get
\begin{equation}
  \frac{1}{C(ln W)^2}\mathcal U^W(\{\mathcal F_{\widetilde{\bm{u}},\widetilde p}\}) \le \,\mathcal{V}^{W}(\{\mathcal F_{\widetilde{\bm{u}},\widetilde p}\}).
 \label{eq:4.5}
\end{equation}
Hence, using \eqref{eq:4.4} and \eqref{eq:4.5} it follows that there exists  constant \(C\) and \(K\) such that
\begin{equation}
  \frac{1}{K}\mathcal V^W(\{\mathcal F_{\widetilde{\bm{u}},\widetilde p}\}) \le \mathcal U^W\{(\mathcal F_{\widetilde{\bm{u}},\widetilde p}\})\,\le C(ln W)^2\mathcal{V}^{W}(\{\mathcal F_{\widetilde{\bm{u}},\widetilde p}\}).
 \label{eq:4.6}
\end{equation}
Thus, the two quadratic forms \(\mathcal U^W(\{\mathcal F_{\widetilde{\bm{u}},\widetilde p}\})\) and \(\mathcal V^W(\{\mathcal F_{\widetilde{\bm{u}},\widetilde p}\})\) are spectrally equivalent. We can use the quadratic form \(\mathcal U^W(\{\mathcal F_{\widetilde{\bm{u}},\widetilde p}\})\) as a preconditioner, which consists of a decoupled set of quadratic forms on each element and gives a block-diagonal matrix, where each diagonal block corresponds to a particular element. It is clear that in each element, the preconditioner corresponds to the quadratic form:

\begin{equation}\label{eq:4.7}
B(\widetilde{\bm{u}},\widetilde{p}) \;=\; \|\widetilde{\bm{u}}\|_{\bm{H}^2(S)}^2 \;+\; \|\widetilde{p}\|_{H^1(S)}^2.
\end{equation}
where, \(\widetilde{\bm{u}}=\widetilde{\bm{u}}(\xi,\eta) \text{ and } \widetilde{p}=\widetilde{p}(\xi,\eta)\) is a polynomial of degree \(W\) in \(\xi,\eta\) separately.\\

As, mentioned in above, PCGM requires the matrix vector product, this action is computed without storing the matrix. 
\section{Error Estimate}
\label{sec:ErrorEstimates_ch4}
To prove the error estimate, we assume that the solution of the interface problem \eqref{eq:3.1}--\eqref{eq:3.2}, is sufficiently smooth in the individual subdomains \( \Omega_{1} \) and \( \Omega_{2} \). Specifically, we assume that \( \bm{u}_{1} \in \bm{H}^{k}(\Omega_{1}) \) and \( \bm{u}_{2} \in \bm{H}^{k}(\Omega_{2}) \) for \( k \geq 2 \). As defined in Section~(\ref{sec:discretization}), \( M_{i}^{l} \) is an analytic map from \( S = (-1,1)^{2} \) to \( \Omega_{i}^{l} \).

Let \(\bm{\lambda}=(\xi,\eta)\) denote
\[
\bm{U}_{i}^{k}(\bm{\lambda}) =
\begin{pmatrix}
U_{1}^{i,k}(\bm{\lambda}) \\
U_{2}^{i,k}(\bm{\lambda})
\end{pmatrix}
= \bm{u}_i(M_{i}^{k}(\bm{\lambda})), \qquad
P_{i}^{k}(\bm{\lambda}) = p_i(M_{i}^{k}(\bm{\lambda})),
\]
for \( \bm{\lambda} \in S \), where \(  1 \le k \le m_i  \) and \( i=1,2 \). Then
\( U_{1}^{i,k}, U_{2}^{i,k} \), and \( P_{i}^{k} \)
are analytic on \( \bar{S} \), and for some constants \( C \) and \( d \), we
have from \cite{babuska}
\begin{align}
|D^{\alpha} U_{j}^{i,k}(\xi,\eta)| &\leq C\, m! d^{m}, \label{eq:4.22} \\
|D^{\alpha} P_{i}^{k}(\xi,\eta)| &\leq C\, m! d^{m}, \label{eq:4.23}
\end{align}
for \( i,j = 1,2 \) and \( |\alpha| = m,\, m = 1,2,\dots \).
\medskip

\begin{theorem}
Let \( \mathcal F_{\widetilde{\bm{z}},\widetilde q} \in \Pi^W \) be the minimizer of the functional \( \mathcal{R}^W(\mathcal F_{\widetilde{\bm{u}},\widetilde p}) \) over all \( \mathcal F_{\widetilde{\bm{u}},\widetilde p} \in \Pi^W \). Then for \( W \) sufficiently large, there exist constants \( C > 0 \), \( b > 0 \), independent of \( W \), such that
\begin{equation*} 
\sum_{i=1}^{2} \sum_{l=1}^{m_i}\left\| \bm{U}_i^l(\bm{\lambda}) - \widetilde{\bm{z}}_i^l(\bm{\lambda}) \right\|_{2,S}^2
+ 
\sum_{i=1}^{2} \sum_{l=1}^{m_i}\left\| P_i^l(\bm{\lambda}) - \widetilde{q}_i^l(\bm{\lambda}) \right\|_{1,S}^2 \\[0.5em] \leq C e^{-b W}.
\end{equation*}

\end{theorem}

\begin{proof}
From the approximation results in \cite{schwab}, there exists a polynomial \( \Phi(\bm{\lambda}) \) of degree \( W \) in each variable separately such that
\begin{equation}
\|U(\bm{\lambda}) - \Phi(\bm{\lambda})\|_{n,S}^2 \leq C_s W^{4 + 2n - 2s} \|U\|_{s,S}^2
\label{eq:4.24}
\end{equation}
for \( 0 \leq n \leq 2 \) and all \( W > s \), where \( C_s = C_1 e^{2s} \).

Hence, there exist polynomials
\(
\bm{\Phi}_i^l(\bm{\lambda}) = \bigl(\Phi_1^{i,l}(\bm{\lambda}), \Phi_2^{i,l}(\bm{\lambda})\bigr)^T
\) and 
\(
\psi_i^l(\bm{\lambda}), i=1,2,
\) such that,

\begin{subequations} \label{eq:4.25}
\begin{align}
\|\bm{U}_i^l(\bm{\lambda}) - \bm{\Phi}_i^l(\bm{\lambda})\|_{2,S}^2 
&\leq C_s W^{-2s + 8} \|\bm{U}_i^l\|_{s,S}^2 \notag \\
&\leq C_s W^{-2s + 8} 
\left( C d^s s! \right)^2
\quad \text{for all } 1 \le l \le m_i, \; i = 1,2.
\label{eq:4.25a} \\[1em]
\|P_i^l(\bm{\lambda}) - \psi_i^l(\bm{\lambda})\|_{1,S}^2 
&\leq C_s W^{-2s + 6} \|\psi_i^l(\bm{\lambda})\|_{s,S}^2 \notag \\
&\leq C_s W^{-2s + 6} 
\left( C d^s s! \right)^2
\quad \text{for all } 1 \le l \le m_i, \; i = 1,2.
\label{eq:4.25b}
\end{align}
\end{subequations}

Define \(\bm{F}_i^l(\bm{\lambda}) = \widetilde{\bm{f}}_i^l(\bm{\lambda})\sqrt{J_i^l}\), and let \(\hat{\bm{F}}_i^l(\bm{\lambda})\) be the polynomial of degree \(2W \) in each variable separately which is the orthogonal projection of \(\bm{F}_i^l\) in \(H^2(S)\) into the space of polynomials of degree \(2W\). Then
\begin{equation}
\|\bm{F}_i^l(\bm{\lambda}) - \hat{\bm{F}}_i^l(\bm{\lambda})\|_{0,S}^2 \leq C_s W^{-2s + 8} \left( C d^s s! \right)^2.
\label{eq:4.26}
\end{equation}

Consider the set of functions \(\{\mathcal F_{\bm{\Phi},\psi}\} = \{\{\bm{\Phi}_1^l(\bm{\lambda})\}_l, \
\{\psi_1^l(\bm{\lambda})\}_l,\{\bm{\Phi}_2^l(\bm{\lambda})\}_l,\
\{\psi_2^l(\bm{\lambda})\}_l\}\). \\
Now we estimate:
\begin{align}
\mathcal{R}^W(\{\mathcal{F}_{\bm{\Phi},\psi}\}) &= 
\sum_{k=1}^{m_1}
\left\| (\mathcal{L}_1^k)^a(\bm{\Phi}_1^k, \psi_1^k) - \hat{\bm{F}}_1^k \right\|_{0, S}^2
+ \sum_{l=1}^{m_2}
\left\| (\mathcal{L}_2^l)^a(\bm{\Phi}_2^l, \psi_2^l)  - \hat{\bm{F}}_2^l\right\|_{0, S}^2 \notag \\[0.5em]
&\quad + \sum_{k=1}^{m_1}
\left\| (\mathcal{D}_1^k)^a \bm{\Phi}_1^k \right\|_{1,S}^2
+ \sum_{l=1}^{m_2}
\left\| (\mathcal{D}_2^l)^a \bm{\Phi}_2^l \right\|_{1,S}^2 \notag \\[0.5em]
&\quad + \sum_{i=1}^2 \sum_{\gamma_s \in \mathcal{F}_i^{\text{int}}}
\bigg(
\left\| \llbracket \bm{\Phi}_i \rrbracket \right\|_{0, \gamma_s}^2
+ \sum_{j=x,y} \left\| \llbracket (\bm{\Phi}_i)^a_{j} \rrbracket \right\|_{\frac{1}{2}\gamma_s}^2
+ \left\| \llbracket \psi_i \rrbracket \right\|_{\frac{1}{2}, \gamma_s}^2
\bigg) \notag \\[0.5em]
&\quad + \sum_{\gamma_s \subset \Gamma_0}
\bigg(
\left\| \llbracket \bm{\Phi}   \rrbracket \right\|_{\frac{3}{2}, \gamma_s}^2
+ \left\| \llbracket (\nu \nabla \bm{\Phi} - \psi \mathbf{I})^a \bm{n} \rrbracket - \hat{\bm{l}}^{m,n} \right\|_{\frac{1}{2}, \gamma_s}^2
\bigg) \notag \\[0.5em]
&\quad + \sum_{\gamma_s \subset \Gamma}
\bigg(
\left\| \bm{\Phi}_2  \right\|_{0, \gamma_s}^2
+ \left\| \left( \frac{\partial \bm{\Phi}_2}{\partial \bm{T}} \right)^a  \right\|_{\frac{1}{2}, \gamma_s}^2
\bigg).
\label{eq:4.27}
\end{align}


As defined earlier, \((\mathcal L_i^l)^a\) is given by,
\begin{equation} \label{eq:4.28}
(\mathcal L_i^l)^a \, \tilde{u} = 
(A_i^l)^a \, \tilde{u}_{\xi \xi} 
+ (B_i^l)^a \, \tilde{u}_{\xi \eta} 
+ (C_i^l)^a \, \tilde{u}_{\eta \eta} 
+ (D_i^l)^a \, \tilde{u}_{\xi} 
+ (E_i^l)^a \, \tilde{u}_{\eta}.
\end{equation}

Here, \((A_i^l)^a\) is the orthogonal projection of \(A_i^l\) in \(H^2(S)\) into the space of polynomials of degree \(W\). The other coefficients of \((\mathcal L_i^l)^a\) are obtained in a similar way. Therefore,
\[
|A_i^l - (A_i^l)^a|^2 \leq C_s \, W^{-2s+8} \, (C d^{s} s!)^2.
\]

Same relation holds for all other coefficients too. Now
\begin{align}\label{eq:4.29}
\left\| (\mathcal{L}_i^l)^a(\bm{\Phi}_i^l, \psi_i^l) - \hat{\bm{F}}_i^l \right\|_{0, S}^2 
&\leq 3 \Big( 
\left\| (\mathcal{L}_i^l)(\bm{U}_i^l, P_i^l)  - (\mathcal{L}_i^l)^a(\bm{U}_i^l, P_i^l) \right\|_{0, S}^2 
\notag \\
&\quad + \left\| (\mathcal{L}_i^l)^a(\bm{U}_i^l, P_i^l)  - (\mathcal{L}_i^l)^a(\bm{\Phi}_i^l, \psi_i^l) \right\|_{0, S}^2 
\notag \\
&\quad + \left\| \bm{F}_i^l - \hat{\bm{F}}_i^l \right\|_{0, S}^2 
\Big)
\notag \\
&\leq K_1 \Big( C_s W^{-2s+8} (C d^s s!)^2 + C_s W^{-2s+6} (C d^s s!)^2 \Big).
\end{align}
Also,
\begin{equation} \label{eq:4.30}
\left\| (\mathcal{D}_i^l)^a \bm{\Phi}_i^l\right\|_{1, S}^2 
\leq K_2C_s W^{-2s + 8} (C d^s s!)^2.
\end{equation}
\noindent
Now, we estimate the following jump terms at inter-element boundary
\begin{equation*} 
\left\| \llbracket \bm{\Phi}_i \rrbracket \right\|^2_{0, \gamma_s}
+  \left\| \llbracket (\bm{\Phi}_i)^a_{x} \rrbracket \right\|^2_{1/2, \gamma_s}
+  \left\| \llbracket (\bm{\Phi}_i)^a_{y} \rrbracket \right\|^2_{1/2, \gamma_s}
\end{equation*}

for any \( \gamma_s \subseteq \Omega_i \), \( i = 1,2 \).

\noindent
The other boundary terms can be handled in a similar fashion. Clearly, \( \gamma_s \) is a side which is common to \( \Omega_i^m \) and \( \Omega_i^n \) for some \( m \) and \( n \). Let us assume that \( \gamma_s \) is the image of the side \( \eta = -1 \) under the mapping \( M^m_i \), which maps \( S \) to \( \Omega_i^m \), and \( \eta = 1 \) under the mapping \( M^n_i \), which maps \( S \) to \( \Omega_i^n \). Then,
\begin{align}
\left\| \llbracket \bm{\Phi_i} \rrbracket \right\|^2_{0, \gamma_s}
&= \int_{-1}^1 
\left( \bm{\Phi}^n_i (\xi,1) - \bm{\Phi}^m_i (\xi,-1) \right)^2 \,d\xi \notag
\\
&\leq K_3 \left( W^{-2s + 8} \, (C_s d^s s!)^2 \right). \label{eq:4.31}
\end{align}

Let \( \xi = \lambda_1 \), \( \eta = \lambda_2 \), \( x = x_1 \), \( y = x_2 \). Then,
\begin{align*}
\left\| \llbracket (\bm{\Phi_i})^a_{x_j} \rrbracket \right\|^2_{1/2, \gamma_s} 
&= \left\| 
\sum_{k=1}^2 (\bm{\Phi_i}^n)_{\lambda_k} ( \widehat{\lambda}_k )^n_{x_j} (\xi,1)
- \sum_{k=1}^2 (\bm{\Phi_i}^m)_{\lambda_k} ( \widehat{\lambda}_k )^m_{x_j} (\xi,-1)
\right\|^2_{1/2, (-1,1)}.
\end{align*}

\medskip

The term \( \widehat{\lambda}_i \) denotes the polynomial approximation of \( \lambda_i \) in \( \bm{\lambda} \)-coordinates (i.e., \( (\xi, \eta) \)) of degree \( W \) in \( \xi \) and \( \eta \) separately, as defined in Theorem~(4.46) of~\cite{schwab}. 


\noindent Moreover, we have
\[
(\bm{U}_i^n)_{x_j}(\xi,1) = (\bm{U}_i^m)_{x_j}( \xi,-1),
\]
and
\[
\| ab \|_{1/2, (-1,1)} \leq \| a \|_{1, \infty, (-1,1)} \| b \|_{1/2, (-1,1)}.
\]
Now \( (\lambda_k)^m_{x_j} \) are analytic functions of \( (\xi, \eta) \), and satisfy
\[
| D^{\alpha_1} D^{\alpha_2}  ((\lambda_k)^m_{x_j}) | \leq C d^s s!
\]
for \( (\xi, \eta) \in S \), \( k = 1,2 \), and \( |\alpha| \leq s \). Thus, we can show
\[
\left\| (\lambda_k)^m_{x_j} - (\widehat{\lambda}_k)^m_{x_j} \right\|_{1/2, S}^2
\leq C_s \left( W^{-2s+8} (C d^s s!)^2 \right),
\]
and
\[
\left| (\widehat{\lambda}_k)^m_{x_j} \right|^2_{1, \infty, (-1,1)} \leq C (\ln W).
\]
Using Sobolev's embedding theorem (Theorem 4.76 in~\cite{schwab} and~\cite{tomar}), and combining all these estimates, we conclude that
\begin{equation} \label{eq:4.32}
\left\| \llbracket (\bm{\Phi_i})^a_{x_j} \rrbracket \right\|^2_{1/2, \gamma_s} 
\leq K_4 \left( C_sW^{-2s+8} (\ln W)(C d^s s!)^2 \right)
\quad \forall j, \; 1 \le j \le 2.
\end{equation}
\noindent


In the same manner, the following estimate also holds:
\begin{align} \label{eq:4.33}
\left\| \llbracket \bm{\Phi} \rrbracket \right\|^2_{0,3/2} 
&+ \left\| \llbracket (\nu \nabla \bm{\Phi} - \psi \mathbf{I})^a \bm{n} \rrbracket 
- \hat{\bm{l}}^{m,n} \right\|^2_{\frac{3}{2}, \gamma_s} \notag \\
&\leq K_5  \Bigg( C_s W^{-2s+8} (\ln W) \ (C d^s s!)^2 
+ C_s W^{-2s+6} \ (C d^s s!)^2 \Bigg).
\end{align}

 In a similar way terms on the boundary can be estimated. 
Adding all the above inequalities from \eqref{eq:4.29}-\eqref{eq:4.33}, we get:
\[
\mathcal{R}^W ( \{ \mathcal{F}_{\bm{\Phi}, \psi}\} ) \leq K_6  \Bigg( C_s W^{-2s+8} (\ln W) \ (C d^s s!)^2 
+ C_s W^{-2s+6} \ (C d^s s!)^2 \Bigg).
\]
Since \( \mathcal F_{\widetilde{\bm{z}},\widetilde q} \) minimizes \( \mathcal{R}^W ( \mathcal F_{\widetilde{\bm{u}},\widetilde p}) \) over all \( \mathcal F_{\widetilde{\bm{u}},\widetilde p} \), we have
\[
\mathcal{R}^W ( \{ \mathcal{F}_{\bm{\Phi},\psi}\} ) 
= \mathcal{R}^W ( \{ \mathcal{F}_{\widetilde{\bm{z}}, \widetilde{q}} \} )
+ \mathcal{V}^W ( \mathcal F_{( \bm{\Phi} - \widetilde{\bm{z}}), (\psi-\widetilde{q}) } ).
\]
Therefore, we conclude that
\[
\mathcal{V}^W ( \{\mathcal F_{ (\bm{\Phi} - \widetilde{\bm{z}}), (\psi-\widetilde{q})} \} ) 
\leq \mathcal{R}^W ( \{ \mathcal{F}_{\bm{\Phi}, \psi} \} ).
\]
Using the stability estimate Theorem~(\ref{thm:stability}), we obtain
\begin{equation}
\begin{aligned}
\sum_{i=1}^2 \sum_{l=1}^{m_i} & \left\| \bm{\Phi}_i^l(\bm{\lambda}) - \widetilde{\bm{z}}_i^l(\bm{\lambda}) \right\|_{2,S}^2 
 + \sum_{i=1}^2 \sum_{l=1}^{m_i} \left\| \psi_i^l(\bm{\lambda}) - \widetilde{q}_i^l(\bm{\lambda}) \right\|_{1,S}^2 \\
& \leq K_7 \Bigl(C_s W^{-2s+8} (\ln W) (C d^s s!)^2 + C_sW^{-2s+6} (C d^s s!)^2 \Bigr).
\end{aligned}
\label{eq:4.34}
\end{equation}

It is easy to show that
\begin{equation}
\begin{aligned}
& \sum_{i=1}^2 \sum_{l=1}^{m_i} 
  \left\| \bm{U}_i^l(\bm{\lambda}) - \bm{\Phi}_i^l(\bm{\lambda}) \right\|_{2,S}^2 
 + \sum_{i=1}^2 \sum_{l=1}^{m_i} 
  \left\| P_i^l(\bm{\lambda}) - \psi_i^l(\bm{\lambda}) \right\|_{1,S}^2 \\
&\leq K_8  \Bigl(C_s W^{-2s+8} (\ln W) (C d^s s!)^2 + C_sW^{-2s+6} (C d^s s!)^2 \Bigr).
\end{aligned}
\label{eq:4.35}
\end{equation}

Combining \eqref{eq:4.34} and \eqref{eq:4.35}, we get
\begin{equation}
\begin{aligned}
&\sum_{i=1}^2 \sum_{l=1}^{m_i} \left\| \bm{U}_i^l(\bm{\lambda}) - \widetilde{\bm{z}}_i^l(\bm{\lambda}) \right\|_{2,S}^2 
+ \sum_{i=1}^2 \sum_{l=1}^{m_i} \left\| P_i^l(\bm{\lambda}) - \widetilde{q}_i^l(\bm{\lambda}) \right\|_{1,S}^2\\
&\leq K  \Bigl(C_s W^{-2s+8} (\ln W) (C d^s s!)^2 + C_sW^{-2s+6} (C d^s s!)^2 \Bigr).
\label{eq:4.36}
\end{aligned}
\end{equation}
Hence, with a proper choice of \( s \) and using Stirling’s formula as in~\cite{tomar}, we can easily prove that there exists a constant \( b > 0 \) such that
\begin{equation}
\begin{aligned}
&\sum_{i=1}^2 \sum_{l=1}^{m_i} \left\| \bm{U}_i^l(\bm{\lambda}) - \widetilde{\bm{z}}_i^l(\bm{\lambda}) \right\|_{2,S}^2 
+ \sum_{i=1}^2 \sum_{l=1}^{m_i} \left\| P_i^l(\bm{\lambda}) - \widetilde{q}_i^l(\bm{\lambda}) \right\|_{1,S}^2\\
&\leq  C e^{-bW}.
\label{eq:4.37}
\end{aligned}
\end{equation}
\end{proof}
\begin{remark}
After obtaining a nonconforming solution, a set of corrections can be applied so that the velocity variable \( \hat{\bm{z}} \) becomes conforming \cite{tomar}. Then \( \hat{\bm{z}} \in \bm{H}^{1}(\Omega) \), and we have the following error estimate:
\begin{align}
\| \bm{u} - \hat{\bm{z}} \|_{1,\Omega} + \| p - q \|_{0,\Omega} \leq C e^{-bW}.
\end{align}
\label{eq:4.38}
\end{remark}



\section{Numerical results}
\label{sec:NumericalResultsCh4}

In this section, we present the numerical results. Exponential convergence
of the numerical scheme is verified through various numerical examples.
The numerical examples include the Stokes interface problem on different
types of domains with different interfaces. We have considered the
interface problems with homogeneous as well as non-homogeneous jump
conditions across the interface. Even though we have derived the theoretical
estimates for the interface problem with Dirichlet boundary conditions,
here we also present the numerical results for the interface problem
with mixed boundary conditions. The numerical solution is exponentially
accurate in all the cases.

Higher-order spectral element functions of degree $W$ are used uniformly
in all the elements of the discretization. The nonconforming solution
is obtained using PCGM at GLL points, and the conforming solution
is obtained using a set of corrections. Let $\hat{\bm{z}}$ be the conforming
approximate solution of the velocity $\bm{u}$ and $q$ be the
approximate solution of the pressure $p.$ $\|E_{\bm{u}}\|_{1}=\frac{\left\Vert \bm{u}-\hat{\bm{z}}\right\Vert _{1}}{\left\Vert \bm{u}\right\Vert _{1}}$
denotes the relative error in $\bm{u}$ in $\bm{H}^{1}$ norm,
$\|E_{p}\|_{0}=\frac{\left\Vert p-q\right\Vert _{0}}{\left\Vert p\right\Vert _{0}}$
denotes the relative error in pressure in $L^{2}$ norm, and $\|E_{c}\|_{0}$
denotes the error in the continuity equation in $L^{2}$ norm. '\textbf{iters}'
denotes the total number of iterations required to reach the desired
accuracy. To ensure the uniqueness of the pressure variable in the
problems with Dirichlet boundary conditions, pressure is specified
to be zero at one point of the domain. 

\subsubsection*{Example 1: With homogeneous jump conditions on the interface}

Let, $\Omega=[0,1]^{2}$, consider the Stokes interface problem with
a straight line interface $\Gamma_{0}=\{(x,y):y=0.5\}$ and with Dirichlet
boundary condition on the boundary. The interface divides the domain
$\Omega$ into two subdomains $\Omega_{1}=\{(x,y)\in\Omega:0<y<0.5\}\text{ and }\Omega_{2}=\{(x,y)\in\Omega:0.5<y<1\}$.
The data is chosen such that the problem has the following exact solution:
\begin{eqnarray*}
u_{1}(x,y)=\begin{cases}
\frac{1}{\nu_{1}}(y-0.5)x^{2}\text{ in }\Omega_{1},\\
\frac{1}{\nu_{2}}(y-0.5)x^{2}\text{ in }\Omega_{2},
\end{cases}u_{2}(x,y)=\begin{cases}
-\frac{1}{\nu_{1}}x(y-0.5)^{2}\text{ }\text{in }\Omega_{1},\\
-\frac{1}{\nu_{2}}x(y-0.5)^{2}\text{ in }\Omega_{2},
\end{cases}
\end{eqnarray*}
\[
p(x,y)=e^{x}-e^{y}+c.
\]
Here, \( c \) is chosen such that pressure is zero at one point in the domain. This exact solution satisfies the homogeneous jump conditions across the interface \( y = 0.5 \).

To obtain the numerical solution, the domain is divided into four
elements such that the discretization matches along the interface
$y=0.5$. The numerical solution for interface problem with different
pairs of viscosity coefficients $\nu_{1}=1,\nu_{2}=0.1,\text{ }\nu_{1}=1,\nu_{2}=0.01$
and $\nu_{1}=1,\nu_{2}=0.001$ is obtained for different values of
$W$. The relative error in $\bm{u}$ in $\bm{H}^{1}$ norm,
relative error in pressure in $L^{2}$ norm, error in the continuity
equation in $L^{2}$ norm and the total number of iterations ("\textbf{iters}")
required to reach the desired accuracy are tabulated in Table $1$
for different values of $W$. 

\begin{table}[H]
~~~~~~~~~~~~~{\small{}}%
\begin{tabular}{|c|c|c|c|c|c|c|c|c|}
\hline 
\multicolumn{1}{|c}{} & \multicolumn{4}{c|}{{\small{}${\nu_{1}}=1$, $\nu_{2}=0.1$}} & \multicolumn{4}{c|}{{\small{}${\nu_{1}}=1$, $\nu_{2}=0.01$}}\tabularnewline
\hline 
{\small{}$\textbf{W}$} & {\small{}$\rVert\textit{E}_{\textbf{u}}\rVert_{1}$} & {\small{}$\rVert\textit{E}_{\textit{p}}\rVert_{0}$} & {\small{}$\rVert\textit{E}_{\textit{c}}\rVert_{0}$} & \textbf{\small{}iters} & {\small{}$\rVert\textit{E}_{\textbf{u}}\rVert_{1}$} & {\small{}$\rVert\textit{E}_{\textit{p}}\rVert_{0}$} & {\small{}$\rVert\textit{E}_{\textit{c}}\rVert_{0}$} & \textbf{\small{}iters}\tabularnewline
\hline 
{\small{}2} & {\small{}3.50E-02} & {\small{}3.01E-01} & {\small{}1.11E-01} & {\small{}30} & {\small{}2.81E-02} & {\small{}4.19E-01} & {\small{}8.15E-01} & {\small{}46}\tabularnewline
\hline 
{\small{}3} & {\small{}6.07E-03} & {\small{}5.07E-02} & {\small{}1.70E-02} & {\small{}98} & {\small{}4.90E-03} & {\small{}4.82E-02} & {\small{}7.64E-02} & {\small{}151}\tabularnewline
\hline 
{\small{}4} & {\small{}7.65E-04} & {\small{}9.39E-03} & {\small{}2.21E-03} & {\small{}228} & {\small{}6.41E-04} & {\small{}6.16E-03} & {\small{}1.12E-02} & {\small{}290}\tabularnewline
\hline 
{\small{}5} & {\small{}6.68E-05} & {\small{}8.57E-04} & {\small{}2.56E-04} & {\small{}399} & {\small{}4.94E-05} & {\small{}4.86E-04} & {\small{}7.59E-04} & {\small{}648}\tabularnewline
\hline 
{\small{}6} & {\small{}5.80E-06} & {\small{}8.25E-05} & {\small{}2.04E-05} & {\small{}744} & {\small{}3.69E-06} & {\small{}4.96E-05} & {\small{}5.20E-05} & {\small{}1050}\tabularnewline
\hline 
{\small{}7} & {\small{}6.39E-07} & {\small{}1.01E-05} & {\small{}2.07E-06} & {\small{}1273} & {\small{}5.01E-07} & {\small{}4.13E-06} & {\small{}6.21E-06} & {\small{}1474}\tabularnewline
\hline 
{\small{}8} & {\small{}4.57E-08} & {\small{}1.17E-06} & {\small{}1.64E-07} & {\small{}2115} & {\small{}2.24E-08} & {\small{}6.14E-07} & {\small{}5.67E-07} & {\small{}2210}\tabularnewline
\hline 
\end{tabular}{\small\par}

{\small{}~~~~~~~~~~~~~~~~~~~~~~~~~~~~~~~~~~~~~~~~~~~}%
\begin{tabular}{|c|c|c|c|c|}
\hline 
\multicolumn{5}{|c|}{{\small{}${\nu_{1}}=1$, $\nu_{2}=0.001$}}\tabularnewline
\hline 
{\small{}$\textbf{W}$} & {\small{}$\rVert\textit{E}_{\textbf{u}}\rVert_{1}$} & {\small{}$\rVert\textit{E}_{\textit{p}}\rVert_{0}$} & {\small{}$\rVert\textit{E}_{\textit{c}}\rVert_{0}$} & \textbf{\small{}iters}\tabularnewline
\hline 
{\small{}2} & {\small{}2.03E-01} & {\small{}1.13E+01} & {\small{}7.25E+01} & {\small{}8}\tabularnewline
\hline 
{\small{}3} & {\small{}1.39E-02} & {\small{}3.99E+00} & {\small{}3.27E+00} & {\small{}77}\tabularnewline
\hline 
{\small{}4} & {\small{}4.03E-03} & {\small{}5.82E-01} & {\small{}1.35E+00} & {\small{}157}\tabularnewline
\hline 
{\small{}5} & {\small{}1.16E-03} & {\small{}2.48E-02} & {\small{}8.63E-02} & {\small{}413}\tabularnewline
\hline 
{\small{}6} & {\small{}4.23E-04} & {\small{}5.48E-03} & {\small{}2.91E-02} & {\small{}662}\tabularnewline
\hline 
{\small{}7} & {\small{}1.43E-04} & {\small{}7.54E-04} & {\small{}2.92E-03} & {\small{}2603}\tabularnewline
\hline 
{\small{}8} & {\small{}1.55E-05} & {\small{}8.88E-05} & {\small{}2.35E-04} & {\small{}4455}\tabularnewline
\hline 
\end{tabular}{\small\par}

\caption{The errors {\small{}$\rVert\textit{E}_{\textbf{u}}\rVert_{1},\rVert\textit{E}_{\textit{p}}\rVert_{0}$}
and {\small{}$\rVert\textit{E}_{\textit{c}}\rVert_{0}$ against $W$ }}

\end{table}

Table 1 shows that the errors decay very fast as $W$ increases,
and the number of iterations is increasing as the ratio of the viscosity
coefficients $\frac{\nu_{1}}{\nu_{2}}$ increases. In the cases of
$\frac{\nu_{1}}{\nu_{2}}=10$ and $\frac{\nu_{1}}{\nu_{2}}=100,$
the approximate solution $\hat{\bm{z}}$ is obtained to an accuracy
of $O(10^{-6})$ with a smaller number of iterations (with $W=6$),
but the iteration count is increased to obtain $\hat{\bm{z}}$ to an
accuracy of $O(10^{-8}).$ In the case of $\frac{\nu_{1}}{\nu_{2}}=1000,$
one can see that the approximate solution $\hat{\bm{z}}$ is obtained
to an accuracy of $O(10^{-4})$ with a smaller number of iterations
(with $W=6$). The graph between the log of relative error and the
degree of polynomial $W$ is shown for $\nu_{1}=1,\nu_{2}=0.1,\text{ and }\nu_{1}=1,\nu_{2}=0.01$
for velocity and pressure variables in figures 5.a and 5.b respectively.
The graph is almost linear, which shows the exponential accuracy of
the proposed method. 
\begin{figure}[H]
\begin{centering}
\subfloat[Log of $\|E_{\bm{u}}\|_{1}$ vs. $W$ for $\upsilon_{2}=0.1,0.01$]{\begin{centering}
\includegraphics[width=8cm,height=5cm]{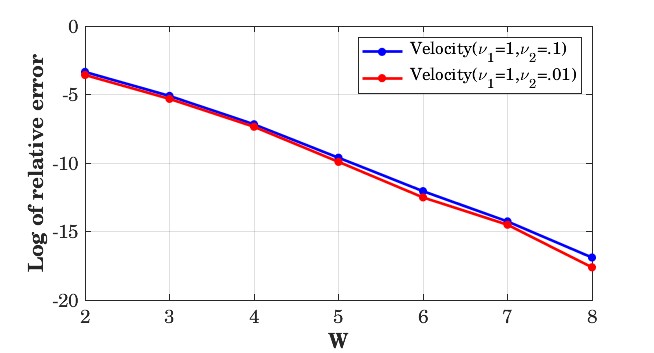}
\par\end{centering}
}\subfloat[Log of $\|E_{p}\|_{0}$ vs. $W$ for $\upsilon_{2}=0.1,0.01$]{\begin{centering}
\includegraphics[width=8cm,height=5cm]{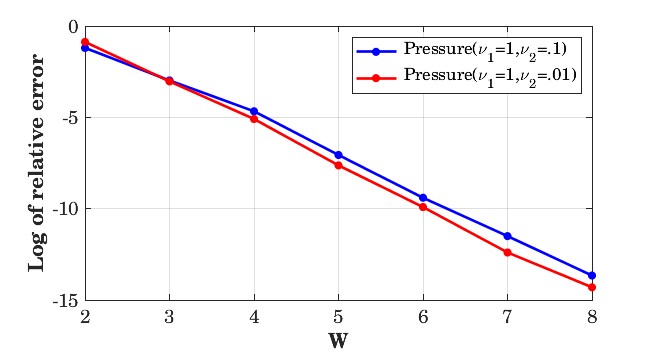}
\par\end{centering}
}
\par\end{centering}
\caption{Log of relative errors against $W$}
\end{figure}

We have interchanged the viscosity coefficient data and obtained the
numerical solution. Table 2 shows the relative errors {\small{}$\rVert\textit{E}_{\textbf{u}}\rVert_{1},\rVert\textit{E}_{\textit{p}}\rVert_{0}$
and the error $\rVert\textit{E}_{\textit{c}}\rVert_{0}$ against $W$
for }$\frac{\nu_{2}}{\nu_{1}}=10$ and $\frac{\nu_{2}}{\nu_{1}}=100.$
The results show the exponential convergence of the numerical scheme. 

\begin{table}[H]
~~~%
\begin{tabular}{|c|c|c|c|c|c|c|c|c|}
\hline 
\multicolumn{1}{|c}{} & \multicolumn{4}{c|}{{\small{}${\nu_{1}}=0.1$, $\nu_{2}=1$}} & \multicolumn{4}{c|}{{\small{}${\nu_{1}}=0.01$, $\nu_{2}=1$}}\tabularnewline
\hline 
{\small{}$W$} & {\small{}$\rVert\textit{E}_{\textbf{u}}\rVert_{1}$} & {\small{}$\rVert\textit{E}_{\textit{p}}\rVert_{0}$} & {\small{}$\rVert\textit{E}_{\textit{c}}\rVert_{0}$} & \textbf{\small{}iters} & {\small{}$\rVert\textit{E}_{\textbf{u}}\rVert_{1}$} & {\small{}$\rVert\textit{E}_{\textit{p}}\rVert_{0}$} & {\small{}$\rVert\textit{E}_{\textit{c}}\rVert_{0}$} & \textbf{\small{}iters}\tabularnewline
\hline 
{\small{}2} & {\small{}4.36E-02} & {\small{}5.73E-01} & {\small{}1.57E-01} & {\small{}32} & {\small{}1.88E-01} & {\small{}7.57E-01} & {\small{}7.16E+00} & 9\tabularnewline
\hline 
{\small{}3} & {\small{}7.84E-03} & {\small{}1.55E-01} & {\small{}2.15E-02} & {\small{}93} & {\small{}5.97E-03} & {\small{}4.46E-01} & {\small{}1.19E-01} & 139\tabularnewline
\hline 
{\small{}4} & {\small{}5.39E-04} & {\small{}2.07E-02} & {\small{}1.49E-03} & {\small{}250} & {\small{}1.74E-03} & {\small{}1.57E-01} & {\small{}4.94E-02} & {\small{}220}\tabularnewline
\hline 
{\small{}5} & {\small{}3.80E-05} & {\small{}1.11E-03} & {\small{}1.38E-04} & {\small{}444} & {\small{}4.86E-04} & {\small{}3.13E-03} & {\small{}3.64E-03} & {\small{}452}\tabularnewline
\hline 
{\small{}6} & {\small{}2.92E-06} & {\small{}2.87E-04} & {\small{}9.60E-06} & {\small{}888} & {\small{}4.90E-05} & {\small{}5.72E-04} & {\small{}5.59E-04} & {\small{}784}\tabularnewline
\hline 
{\small{}7} & {\small{}1.47E-07} & {\small{}1.83E-05} & {\small{}5.12E-07} & {\small{}1528} & {\small{}6.77E-07} & {\small{}8.26E-05} & {\small{}8.72E-06} & {\small{}1429}\tabularnewline
\hline 
\end{tabular}

\caption{Errors against $W$ for $\nu_{1}=0.1,0.01$}

\end{table}

\subsubsection*{Example 2: With non-homogeneous jump condition on the interface}

Let $\Omega=[0,1]^{2}.$ Consider the Stokes interface problem with
a straight line interface $\Gamma_{0}=\{(x,y):y=0.5\}$ and with Dirichlet
boundary condition on the boundary. The interface divides the domain
$\Omega$ into two subdomains $\Omega_{1}=\{(x,y)\in\Omega:0<y<0.5\}\text{ and }\Omega_{2}=\{(x,y)\in\Omega:0.5<y<1\}$.
The data is chosen such that the problem has the following exact solution:
\begin{eqnarray*}
u_{1}(x,y)=\begin{cases}
\frac{1}{\nu_{1}}(y-0.5)x^{2}\text{ in }\Omega_{1},\\
\frac{1}{\nu_{2}}(y-0.5)x^{2}\text{ in }\Omega_{2},
\end{cases}u_{2}(x,y)=\begin{cases}
-\frac{1}{\nu_{1}}x(y-0.5)^{2}\text{ in }\Omega_{1},\\
-\frac{1}{\nu_{2}}x(y-0.5)^{2}\text{ in }\Omega_{2},
\end{cases}
\end{eqnarray*}
\[
p(x,y)=\begin{cases}
2xy+x^{2}\text{ in }\Omega_{1},\\
2xy+x^{2}-3\text{ in }\Omega_{2}.
\end{cases}
\]
\vspace{0.2cm}
Here, the solution satisfies non-homogeneous jump condition ($[(\nu\nabla\bm{u}-pI)\mathbf{n}]$)
on the interface $y=0.5$. The numerical solution is obtained for
$\nu_{1}=1,\nu_{2}=0.1.$

\begin{wraptable}{o}{0.5\columnwidth}%
{\small{}}%
\begin{tabular}{|c|c|c|c|c|}
\hline 
\multicolumn{5}{|c|}{{\small{}$\nu_{1}=1,\nu_{2}=0.1$}}\tabularnewline
\hline 
{\small{}$\textbf{W}$} & {\small{}$\rVert\textit{E}_{\textbf{u}}\rVert_{1}$} & {\small{}$\rVert\textit{E}_{\textit{p}}\rVert_{0}$} & {\small{}$\rVert\textit{E}_{\textit{c}}\rVert_{0}$} & \textbf{\small{}iters}\tabularnewline
\hline 
{\small{}2} & {\small{}3.46E-02} & {\small{}5.42E-02} & {\small{}1.01E-01} & {\small{}32}\tabularnewline
\hline 
{\small{}4} & {\small{}1.62E-04} & {\small{}6.76E-04} & {\small{}4.77E-04} & {\small{}250}\tabularnewline
\hline 
{\small{}6} & {\small{}1.48E-06} & {\small{}4.91E-06} & {\small{}4.64E-06} & {\small{}847}\tabularnewline
\hline 
{\small{}8} & {\small{}1.06E-08} & {\small{}7.22E-08} & {\small{}3.90E-08} & {\small{}2318}\tabularnewline
\hline 
{\small{}10} & {\small{}7.99E-11} & {\small{}7.53E-10} & {\small{}3.21E-10} & {\small{}4642}\tabularnewline
\hline 
\end{tabular}{\small\par}

\caption{Errors for {\small{}$\nu_{1}=1,\nu_{2}=0.1$} }
\end{wraptable}%
The domain is divided into four elements such that the discretization
matches along the interface $y=0.5$. The numerical solution is obtained
for different values of $W$. The relative errors $\|E_{\bm{u}}\|_{1},\|E_{p}\|_{0},$
and the error in the continuity equation $\|E_{c}\|_{0}$ and the
total number of iterations required to reach the desired accuracy
are tabulated in Table $3$ for different values of $W$. Similar
results are expected for $\nu_{1}=1,\nu_{2}=0.01$ with an increase
in the number of iterations. One can see the approximate solution
$\hat{\bm{z}}$ is obtained to an accuracy of $O(10^{-6})$ with a smaller
number of iterations (with $W=6$). The numerical scheme is able to
approximate $\bm{u}$ to an accuracy of $O(10^{-11}),$ $p$ to
an accuracy of $O(10^{-10})$ with an increase in the iteration count.
The decay of the error in the continuity equation $\|E_{c}\|_{0}$
shows that the numerical scheme has good mass conservation property. 

\begin{wrapfigure}{o}{0.5\columnwidth}%
$\!\!\!\!\!\!$\includegraphics[scale=0.25]{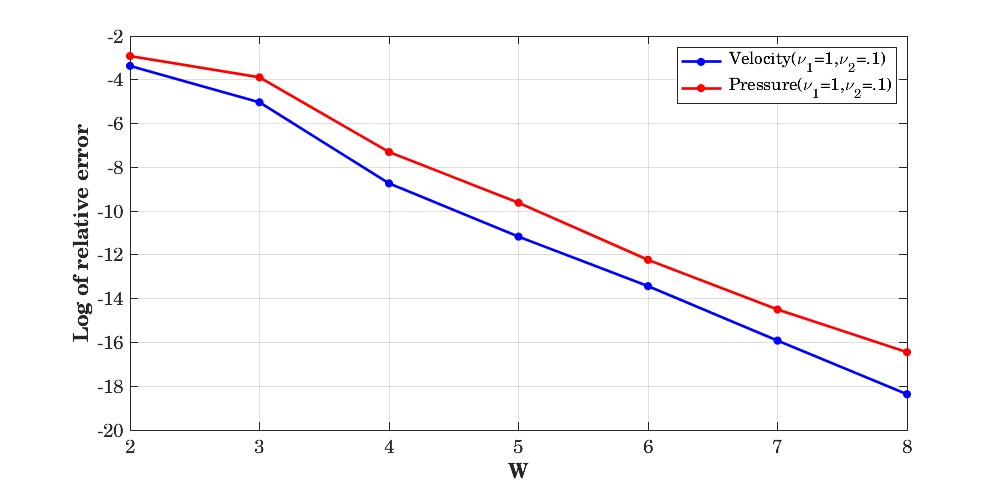}

\caption{Log of errors $\|E_{\bm{u}}\|_{1},\|E_{p}\|_{0}$ vs. $W$ }
\end{wrapfigure}%
The graph between the log of relative errors $\|E_{\bm{u}}\|_{1},\|E_{p}\|_{0}$
and the degree of polynomial $W$ is shown for $\nu_{1}=1,\nu_{2}=0.1$
in figure 6. The graph is almost linear, which shows the exponential
accuracy of the method.

\subsubsection*{Example 3: Stokes interface problem on an annular domain}

Consider the stokes interface problem on the annular domain $\Omega=\{(r,\theta):1\le r\le2,0\le\theta\le\frac{\pi}{2}\}$
having curved interface $\Gamma_{0}=\{(r_{0},\theta):r_{0}=1.5,0\le\theta\le\frac{\pi}{2}\}$
that separates $\Omega$ into two subdomains $\Omega_{1}=\{(x,y):1<x^{2}+y^{2}<2.25\}\text{ and }\Omega_{2}=\{(x,y):2.25<x^{2}+y^{2}<4\}$.
We consider Dirichlet boundary condition on the boundary.

\begin{wraptable}{o}{0.5\columnwidth}%
{\small{}}%
\begin{tabular}{|c|c|c|c|c|}
\hline 
\multicolumn{5}{|c|}{{\small{}$\nu_{1}=1,\nu_{2}=0.1$}}\tabularnewline
\hline 
{\small{}$\textbf{W}$} & {\small{}$\rVert\textit{E}_{\textbf{u}}\rVert_{1}$} & {\small{}$\rVert\textit{E}_{\textit{p}}\rVert_{0}$} & {\small{}$\rVert\textit{E}_{\textit{c}}\rVert_{0}$} & \textbf{\small{}iters}\tabularnewline
\hline 
{\small{}2} & {\small{}9.37E-01} & {\small{}9.86E-01} & {\small{}2.35E+00} & {\small{}1}\tabularnewline
\hline 
{\small{}4} & {\small{}2.85E-02} & {\small{}3.16E-01} & {\small{}4.20E-01} & {\small{}120}\tabularnewline
\hline 
{\small{}6} & {\small{}2.72E-04} & {\small{}2.88E-03} & {\small{}1.07E-02} & {\small{}725}\tabularnewline
\hline 
{\small{}8} & {\small{}2.30E-06} & {\small{}4.07E-05} & {\small{}9.54E-05} & {\small{}2644}\tabularnewline
\hline 
\end{tabular}\caption{Errors for $\nu_{1}=1,\nu_{2}=0.1$}
\end{wraptable}%
The data is chosen such that the problem has the following exact solution:
\begin{eqnarray*}
u_{1}(x,y)=\begin{cases}
\frac{-1}{\nu_{1}}\,\ y\,\ \sin(2.25-x^{2}-y^{2})\text{ in }\Omega_{1},\\
\frac{-1}{\nu_{2}}\,\ y\,\ \sin(2.25-x^{2}-y^{2})\text{ in }\Omega_{2},
\end{cases}\\
u_{2}(x,y)=\begin{cases}
\frac{1}{\nu_{1}}\,\ x\,\sin(2.25-x^{2}-y^{2})\text{ in }\Omega_{1},\\
\frac{1}{\nu_{2}}\,\ x\,\sin(2.25-x^{2}-y^{2})\text{ in }\Omega_{2},
\end{cases}
\end{eqnarray*}
\[
p(x,y)=e^{x+y}-e^{2}.
\]
Here, the solution satisfies the homogeneous jump condition on the
interface $\Gamma_{0}$. The annular domain is divided into four elements
such that the discretization matches along the interface $\Gamma_{0}$
(see figure 7.a). The interface is completely resolved using blending
elements \cite{GO}. The numerical solution is obtained for different
values of $W$ for $\nu_{1}=1,\nu_{2}=0.1$. The relative error in
$\bm{u}$ in $\bm{H}^{1}$ norm, relative error in pressure
in $L^{2}$ norm, error in the continuity equation in $L^{2}$ norm
and the total number of iterations are tabulated in Table $4$ for
different values of $W$. The graph between the log of the relative
error and the degree of polynomial $W$ is shown for $\nu_{1}=1,\nu_{2}=0.1$
in Figure 7.b. The graph is almost linear, which shows the exponential
accuracy of the method. 
\begin{figure}[H]
\subfloat[Discretization of the annular domain]{\includegraphics[width=8cm,height=5cm]{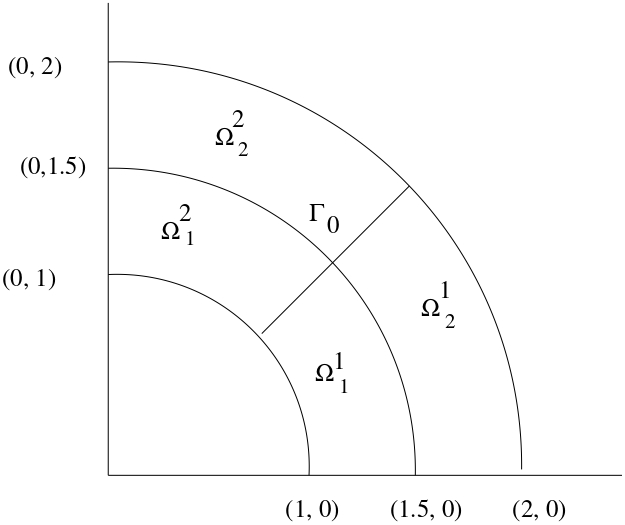}

}~~~~~~\subfloat[Log of the relative errors for $\nu_{1}=1,\nu_{2}=0.1$]{\includegraphics[width=8cm,height=5cm]{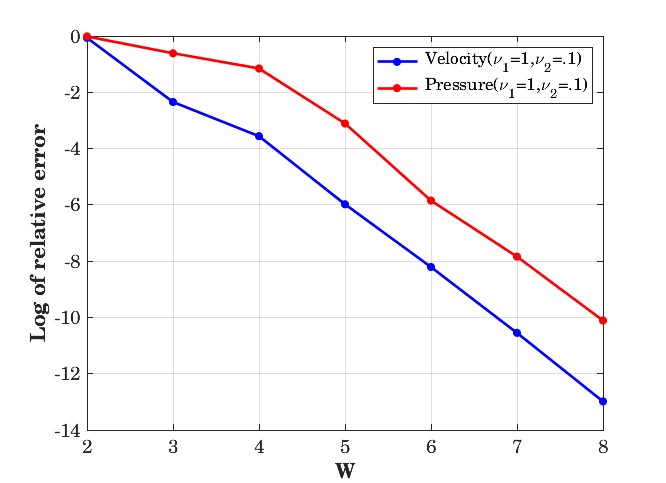}

}

\caption{Domain with curved interface and error in the numerical approximation }
\end{figure}

\subsubsection*{Example 4: Stokes interface problem with circular interface}

Consider the stokes interface problem on the domain $\Omega=[-1,1]^{2}$
having interface $\Gamma_{0}=\{(x,y):x^{2}+y^{2}=0.25\}$ that separates
$\Omega$ into two subdomains $\Omega_{1}=\{(x,y):x^{2}+y^{2}<0.25\}\text{ and }\Omega_{2}=\Omega\setminus\overline{\Omega}_{1}$
with Dirichlet boundary condition on the boundary of the domain $\Omega$.
The data is chosen such that the problem has the following exact solution:
\begin{eqnarray*}
u_{1}(x,y)=\begin{cases}
\frac{1}{\nu_{1}}\,\ y(x^{2}+y^{2}-0.25)\text{ in }\Omega_{1},\\
\frac{1}{\nu_{2}}\,\ y(x^{2}+y^{2}-0.25)\text{ in }\Omega_{2},
\end{cases}u_{2}(x,y)=\begin{cases}
-\frac{1}{\nu_{1}}\,\ x(x^{2}+y^{2}-0.25)\text{ in }\Omega_{1},\\
-\frac{1}{\nu_{2}}\,\ x(x^{2}+y^{2}-0.25)\text{ in }\Omega_{2},
\end{cases}
\end{eqnarray*}
\[
p(x,y)=x^{2}-y^{2}.
\]
The solution satisfies the homogeneous jump condition on the interface
$\Gamma_{0}$. The domain is divided into nine elements such that
the discretization matches along the interface $\Gamma_{0}$(very
similar to the discretization shown in Fig. 2). The numerical solution
is obtained for different values of $W$. The relative error in $\textbf{\textit{u}}$
in $\bm{H}^{1}$ norm, the relative error in pressure in $L^{2}$
norm, and the error in the continuity equation in $L^{2}$ norm are
tabulated in Table $5$ for $\nu_{2}=1,\nu_{1}=0.1\text{ and }\nu_{2}=1,\nu_{1}=0.01$
for different values of $W$. The graph between the log of relative
errors $\|E_{\bm{u}}\|_{1},\|E_{p}\|_{0},$ and the degree of
polynomial $W$ is shown for $\nu_{2}=1,\nu_{1}=0.1$ and $\nu_{2}=1,\nu_{1}=0.01$
in Figures 8.a and 8.b respectively. The graphs show the exponential
accuracy of the method. 

\begin{table}[H]
~~~~~~{\small{}}%
\begin{tabular}{|c|c|c|c|c|c|c|c|c|}
\hline 
\multicolumn{1}{|c}{} & \multicolumn{4}{c|}{{\small{}${\nu_{1}}=0.1$, $\nu_{2}=1$}} & \multicolumn{4}{c|}{{\small{}${\nu_{1}}=0.01$, $\nu_{2}=1.0$}}\tabularnewline
\hline 
{\small{}$\textbf{W}$} & {\small{}$\rVert\textit{E}_{\textbf{u}}\rVert_{1}$} & {\small{}$\rVert\textit{E}_{\textit{p}}\rVert_{0}$} & {\small{}$\rVert\textit{E}_{\textit{c}}\rVert_{0}$} & \textbf{\small{}iters} & {\small{}$\rVert\textit{E}_{\textbf{u}}\rVert_{1}$} & {\small{}$\rVert\textit{E}_{\textit{p}}\rVert_{0}$} & {\small{}$\rVert\textit{E}_{\textit{c}}\rVert_{0}$} & \textbf{\small{}iters}\tabularnewline
\hline 
{\small{}2} & {\small{}8.95E-01} & {\small{}9.87E-01} & {\small{}3.69E-01} & {\small{}1} & {\small{}9.92E-01} & {\small{}9.86E-01} & {\small{}3.95E-01} & {\small{}1}\tabularnewline
\hline 
{\small{}3} & {\small{}1.58E-01} & {\small{}4.27E-01} & {\small{}3.39E-01} & {\small{}86} & {\small{}6.65E-01} & {\small{}1.59E-01} & {\small{}2.65E-01} & {\small{}313}\tabularnewline
\hline 
{\small{}4} & {\small{}2.21E-02} & {\small{}9.72E-02} & {\small{}1.06E-01} & {\small{}352} & {\small{}5.62E-02} & {\small{}6.87E-02} & {\small{}1.38E-01} & {\small{}1348}\tabularnewline
\hline 
{\small{}5} & {\small{}2.78E-03} & {\small{}1.62E-02} & {\small{}1.47E-02} & {\small{}727} & {\small{}3.31E-03} & {\small{}8.65E-03} & {\small{}1.77E-02} & {\small{}2360}\tabularnewline
\hline 
{\small{}6} & {\small{}4.13E-04} & {\small{}2.22E-03} & {\small{}2.19E-03} & {\small{}1621} & {\small{}3.89E-04} & {\small{}1.38E-03} & {\small{}2.40E-03} & {\small{}3843}\tabularnewline
\hline 
{\small{}7} & {\small{}5.58E-05} & {\small{}3.20E-04} & {\small{}3.00E-04} & {\small{}2861} & {\small{}5.89E-05} & {\small{}1.92E-04} & {\small{}2.96E-04} & {\small{}5629}\tabularnewline
\hline 
\end{tabular}{\small\par}

\caption{Relative errors {\small{}$\rVert\textit{E}_{\textbf{u}}\rVert_{1}$,$\rVert\textit{E}_{\textit{p}}\rVert_{0}$
and $\rVert\textit{E}_{\textit{c}}\rVert_{0}$} for $\upsilon_{1}=0.1,0.01$}
\end{table}

\begin{figure}[H]
\begin{centering}
\subfloat[Decay of the relative errors in $\bm{u}$ and $p$ for $\nu_{1}=0.1$]{\begin{centering}
\includegraphics[width=8.5cm,height=5cm]{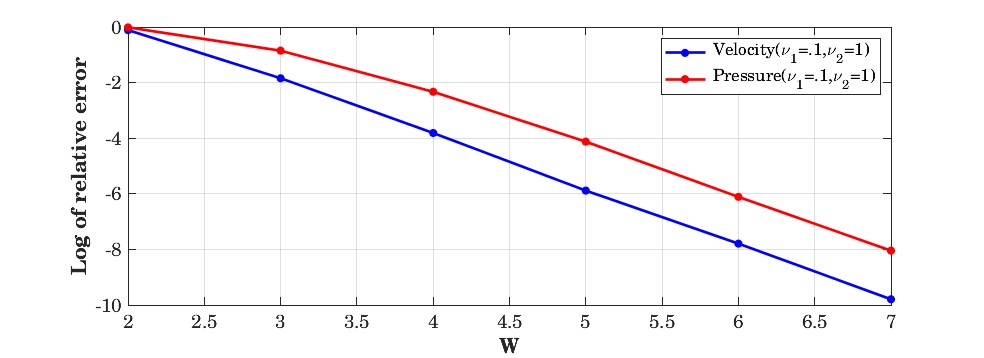}
\par\end{centering}

}\subfloat[Decay of the relative error in $\bm{u}$ and $p$ for $\nu_{1}=0.01$]{\begin{centering}
\includegraphics[width=8cm,height=5cm]{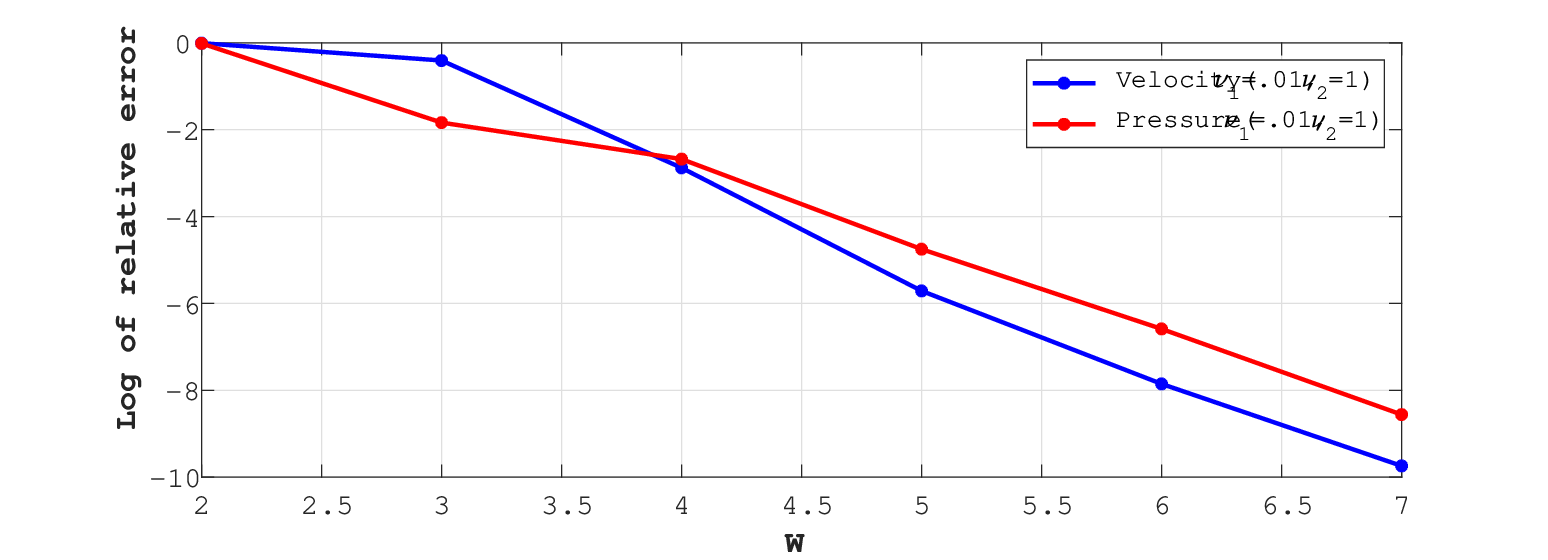}
\par\end{centering}

}
\par\end{centering}
\caption{Log of relative errors against $W$}
\end{figure}

\subsubsection*{Example 5: Interface problem with mixed boundary conditions}

Here, we consider the Stokes interface problem on the domain $\Omega=[-1,1]^{2}$
having a circular interface $\Gamma_{0}=\{(x,y):x^{2}+y^{2}=0.25\}$
that separates $\Omega$ into two subdomains $\Omega_{1}=\{(x,y):x^{2}+y^{2}<0.25\}\text{ and }\Omega_{2}=\Omega\setminus\overline{\Omega}_{1}.$
We consider the Neumann type boundary condition $\frac{\partial\bm{u}}{\partial\mathbf{n}}-p\mathbf{n}$
on the side $y=-1$ and the Dirichlet boundary condition on all the
other sides of the boundary of the domain. The data is chosen such
that the exact solution is the same as the one considered in Example
$4.$ 

We divided the domain into nine elements such that the discretization
matches along the interface $\Gamma_{0}$. Here we have considered
piecewise constant viscosity coefficients $\nu_{1}=0.1,\nu_{2}=1,\text{ and }\nu_{1}=0.01,\nu_{2}=1$
and obtained the numerical solution for different values of $W$.
Table 6 presents the relative errors in $\bm{u},p$ and the error
in the continuity equation for $\nu_{1}=0.1,\nu_{2}=1,\text{ and }\nu_{1}=0.01,\nu_{2}=1.$
The results confirm the exponential convergence in velocity and pressure
variables and the mass conservation property of the numerical scheme. 

\begin{table}[H]
~~~~~~~~{\small{}}%
\begin{tabular}{|c|c|c|c|c|c|c|c|c|}
\hline 
\multicolumn{1}{|c}{} & \multicolumn{4}{c|}{{\small{}${\nu_{1}}=0.1$, $\nu_{2}=1$}} & \multicolumn{4}{c|}{{\small{}${\nu_{1}}=0.01$, $\nu_{2}=1$}}\tabularnewline
\hline 
{\small{}$\textbf{W}$} & {\small{}$\rVert\textit{E}_{\textbf{u}}\rVert_{1}$} & {\small{}$\rVert\textit{E}_{\textit{p}}\rVert_{0}$} & {\small{}$\rVert\textit{E}_{\textit{c}}\rVert_{0}$} & \textbf{\small{}iters} & {\small{}$\rVert\textit{E}_{\textbf{u}}\rVert_{1}$} & {\small{}$\rVert\textit{E}_{\textit{p}}\rVert_{0}$} & {\small{}$\rVert\textit{E}_{\textit{c}}\rVert_{0}$} & \textbf{\small{}iters}\tabularnewline
\hline 
{\small{}2} & {\small{}8.85E-01} & {\small{}9.82E-01} & {\small{}4.87E-01} & {\small{}1} & {\small{}9.93E-01} & {\small{}9.81E-01} & {\small{}5.10E-01} & {\small{}1}\tabularnewline
\hline 
{\small{}3} & {\small{}8.20E-02} & {\small{}2.34E-01} & {\small{}2.78E-01} & {\small{}168} & {\small{}6.81E-01} & {\small{}1.65E-01} & {\small{}2.74E-01} & {\small{}299}\tabularnewline
\hline 
{\small{}4} & {\small{}2.76E-02} & {\small{}1.82E-01} & {\small{}1.10E-01} & {\small{}381} & {\small{}5.64E-02} & {\small{}8.50E-02} & {\small{}1.41E-01} & {\small{}1436}\tabularnewline
\hline 
{\small{}5} & {\small{}4.25E-03} & {\small{}5.23E-02} & {\small{}9.55E-03} & {\small{}1396} & {\small{}2.81E-03} & {\small{}8.25E-03} & {\small{}1.90E-02} & {\small{}2762}\tabularnewline
\hline 
{\small{}6} & {\small{}5.05E-04} & {\small{}5.15E-03} & {\small{}1.85E-03} & {\small{}2378} & {\small{}2.98E-04} & {\small{}1.43E-03} & {\small{}1.84E-03} & {\small{}4650}\tabularnewline
\hline 
{\small{}7} & {\small{}3.58E-05} & {\small{}3.72E-04} & {\small{}1.55E-04} & {\small{}5078} & {\small{}3.40E-05} & {\small{}1.50E-04} & {\small{}3.46E-04} & {\small{}6784}\tabularnewline
\hline 
\end{tabular}{\small\par}

\caption{{\small{}$\rVert\textit{E}_{\textbf{u}}\rVert_{1}$,$\rVert\textit{E}_{\textit{p}}\rVert_{0}$
and $\rVert\textit{E}_{\textit{c}}\rVert_{0}$ for different $W$ }}
\end{table}

\textbf{$\!\!\!\!\!\!\!\!\!\!$}\textbf{\textit{Remark:}}\textit{
In all the above examples, we have considered small viscosity coefficients
in one of the subdomains $\Omega_{1}$ or $\Omega_{2}$ (either $\nu_{1}=0.1/0.01,\nu_{2}=1$
or $\nu_{1}=1,\nu_{2}=0.1/0.01/0.001)$. In all the cases, we have
used the quadratic form given in Eq. \eqref{eq:4.3} as a preconditioner. This
preconditioner is not giving satisfactory results in terms of accuracy
and efficiency when the viscosity coefficients are large (for example,
either $\nu_{1}=100/1000$ or $\nu_{2}=100/1000$). So we have modified
the quadratic form \eqref{eq:4.3} as in these cases and obtained the numerical
results. Here we present the numerical results for the Stokes interface
problem with large piecewise viscocity coefficients with the modified
precondiners. }

\subsubsection*{Numerical results with different preconditioners}

First, we consider the Stokes interface problem stated in Example
2, but with different viscosity coefficients: $\nu_{1}=10,\nu_{2}=1.0$
and $\nu_{1}=100,\nu_{2}=1.0.$ In this case, we have used the following
quadratic form as a preconditioner and obtained the numerical solution
for different values of $W.$
\begin{eqnarray*}
\sum_{k=1}^{m_{1}}\text{\ensuremath{\nu_{1}^{2}}}||\tilde{\bm{u}}_{1}^{k}(\xi,\eta)||_{2,S}^{2}+\sum_{k=1}^{m_{1}}||\tilde{p}_{1}^{k}(\xi,\eta)||_{1,S}^{2}+\sum_{l=1}^{m_{2}}\text{\ensuremath{\nu_{2}^{2}}}||\tilde{\bm{u}}_{2}^{l}(\xi,\eta)||_{2,S}^{2}+\sum_{l=1}^{m_{2}}||\tilde{p}_{2}^{l}(\xi,\eta)||_{1,S}^{2}.
\end{eqnarray*}

We have used the same discretization as in Example 2 and obtained
the numerical solution using PCGM. Table 7 presents the relative errors
in $\bm{u},p,$ and error in the continuity equation for different
values of $W.$ The results show the accuracy of the method, and the
error in the continuity equation shows the mass conservation property
of the numerical method.

\begin{table}[H]
~~~~~~{\small{}}%
\begin{tabular}{|c|c|c|c|c|c|c|c|c|}
\hline 
\multicolumn{1}{|c}{} & \multicolumn{4}{c|}{{\small{}${\nu_{1}}=10$, $\nu_{2}=1$}} & \multicolumn{4}{c|}{{\small{}${\nu_{1}}=100$, $\nu_{2}=1$}}\tabularnewline
\hline 
{\small{}$\textbf{W}$} & {\small{}$\rVert\textit{E}_{\textbf{u}}\rVert_{1}$} & {\small{}$\rVert\textit{E}_{\textit{p}}\rVert_{0}$} & {\small{}$\rVert\textit{E}_{\textit{c}}\rVert_{0}$} & \textbf{\small{}iters} & {\small{}$\rVert\textit{E}_{\textbf{u}}\rVert_{1}$} & {\small{}$\rVert\textit{E}_{\textit{p}}\rVert_{0}$} & {\small{}$\rVert\textit{E}_{\textit{c}}\rVert_{0}$} & \textbf{\small{}iters}\tabularnewline
\hline 
{\small{}2} & {\small{}5.72E-01} & {\small{}6.83E-01} & {\small{}1.37E-01} & {\small{}2} & {\small{}5.19E-01} & {\small{}6.83E-01} & {\small{}1.08E-01} & {\small{}2}\tabularnewline
\hline 
{\small{}3} & {\small{}2.42E-01} & {\small{}3.82E-01} & {\small{}5.56E-02} & 24 & {\small{}1.79E-01} & {\small{}2.80E-01} & {\small{}3.75E-02} & 43\tabularnewline
\hline 
{\small{}4} & {\small{}1.65E-02} & {\small{}7.71E-02} & {\small{}4.60E-03} & {\small{}261} & {\small{}8.17E-03} & {\small{}8.79E-02} & {\small{}3.08E-03} & {\small{}270}\tabularnewline
\hline 
{\small{}5} & {\small{}2.80E-03} & {\small{}1.12E-02} & {\small{}8.86E-04} & {\small{}689} & {\small{}1.85E-03} & {\small{}5.04E-02} & {\small{}5.70E-04} & {\small{}859}\tabularnewline
\hline 
{\small{}6} & {\small{}5.75E-04} & {\small{}1.78E-03} & {\small{}2.52E-04} & {\small{}1548} & {\small{}6.90E-04} & {\small{}2.37E-02} & {\small{}2.02E-04} & {\small{}2174}\tabularnewline
\hline 
\end{tabular}{\small\par}

\caption{Relative errors {\small{}$\rVert\textit{E}_{\textbf{u}}\rVert_{1}$,$\rVert\textit{E}_{\textit{p}}\rVert_{0}$
and $\rVert\textit{E}_{\textit{c}}\rVert_{0}$} for $\upsilon_{1}=10,100$}
\end{table}

Next, we consider the interface problem stated in Example 4, but with
different viscosity coefficients: $\nu_{1}=100,\nu_{2}=1.0$ and $\nu_{1}=1000,\nu_{2}=1.0.$
In this case, we have used the following quadratic form as a preconditioner
and obtained the numerical solution for different values of $W.$
\begin{eqnarray*}
\sum_{k=1}^{m_{1}}\text{\ensuremath{\nu_{1}^{3}}}||\tilde{\bm{u}}_{1}^{k}(\xi,\eta)||_{2,S}^{2}+\sum_{k=1}^{m_{1}}||\tilde{p}_{1}^{k}(\xi,\eta)||_{1,S}^{2}+\sum_{l=1}^{m_{2}}\text{\ensuremath{\nu_{2}^{3}}}||\tilde{\bm{u}}_{2}^{l}(\xi,\eta)||_{2,S}^{2}+\sum_{l=1}^{m_{2}}||\tilde{p}_{2}^{l}(\xi,\eta)||_{1,S}^{2}.
\end{eqnarray*}

\begin{table}[H]
~~~~~~{\small{}}%
\begin{tabular}{|c|c|c|c|c|c|c|c|c|}
\hline 
\multicolumn{1}{|c}{} & \multicolumn{4}{c|}{{\small{}${\nu_{1}}=100$, $\nu_{2}=1$}} & \multicolumn{4}{c|}{{\small{}${\nu_{1}}=1000$, $\nu_{2}=1$}}\tabularnewline
\hline 
{\small{}$\textbf{W}$} & {\small{}$\rVert\textit{E}_{\textbf{u}}\rVert_{1}$} & {\small{}$\rVert\textit{E}_{\textit{p}}\rVert_{0}$} & {\small{}$\rVert\textit{E}_{\textit{c}}\rVert_{0}$} & \textbf{\small{}iters} & {\small{}$\rVert\textit{E}_{\textbf{u}}\rVert_{1}$} & {\small{}$\rVert\textit{E}_{\textit{p}}\rVert_{0}$} & {\small{}$\rVert\textit{E}_{\textit{c}}\rVert_{0}$} & \textbf{\small{}iters}\tabularnewline
\hline 
{\small{}2} & {\small{}8.71E-01} & {\small{}9.86E-01} & {\small{}3.96E-01} & {\small{}1} & {\small{}8.71E-01} & {\small{}9.86E-01} & {\small{}3.95E-01} & {\small{}1}\tabularnewline
\hline 
{\small{}3} & {\small{}4.92E-02} & {\small{}3.42E-01} & {\small{}1.40E-01} & 30 & {\small{}5.05E-02} & {\small{}2.66E-01} & {\small{}1.18E-01} & 28\tabularnewline
\hline 
{\small{}4} & {\small{}1.39E-02} & {\small{}1.23E-01} & {\small{}4.94E-02} & 170 & {\small{}2.56E-02} & {\small{}7.62E-02} & {\small{}4.75E-02} & {\small{}86}\tabularnewline
\hline 
{\small{}5} & {\small{}4.90E-03} & {\small{}5.66E-02} & {\small{}1.13E-02} & 343 & {\small{}3.98E-03} & {\small{}3.53E-02} & {\small{}1.37E-02} & {\small{}274}\tabularnewline
\hline 
{\small{}6} & {\small{}6.62E-04} & {\small{}1.87E-02} & {\small{}1.48E-03} & {\small{}1086} & {\small{}6.26E-04} & {\small{}1.55E-02} & {\small{}1.61E-03} & {\small{}1586}\tabularnewline
\hline 
\end{tabular}{\small\par}

\caption{Errors for $\nu_{1}=100,1000$}
\end{table}

By discretizing the domain with a circular interface into nine elements,
as in Example 4, we have obtained the relative errors {\small{}$\rVert\textit{E}_{\textbf{u}}\rVert_{1}$,$\rVert\textit{E}_{\textit{p}}\rVert_{0}$
and $\rVert\textit{E}_{\textit{c}}\rVert_{0}$} for different values
of $W.$ Table 8 presents these results for $\nu_{1}=100,\nu_{2}=1.0$
and $\nu_{1}=1000,\nu_{2}=1.0.$ The results show the exponential
accuracy of the numerical method.

\section{Conclusions}
 \label{sec:conclusion}

An exponentially accurate nonconforming spectral element method for
the Stokes interface problem with a smooth interface is presented.
The numerical results confirm the exponential accuracy of the numerical
method. This method has many advantages. The method is nonconforming,
and the interface is completely resolved using blending elements.
The inf-sup condition is not required to choose the approximation
spaces for velocity and pressure variables, so the same order spectral
element functions are used for both velocity and pressure variables.
The numerical formulation always leads to a symmetric, positive-definite
linear system. The method also has good mass conservation property.
The method also works for various boundary conditions different from
the Dirichlet boundary conditions. 

Numerical results for the Stokes interface problems with different
sets of piecewise constant viscosity coefficients and different types
of smooth interfaces are presented. One can see that the iteration
count is high when the ratio of viscosity coefficients is high. Different
types of preconditioners are used in some special cases. So there
is a need to develop an efficient preconditioner that works uniformly
in all cases. The development of an efficient preconditioner is under
investigation. The extension of the method to the three-dimensional
Stokes interface problem is ongoing work. 

\subsubsection*{Declarations }

\textbf{Conflict of interests:} All authors declare no conflict of
interests. \\
\\
\textbf{Data Availability:} The data related to this manuscript will
be available upon request.

\end{document}